\documentclass[11pt]{article}
\title{The spectral gap of random regular graphs}
\author{
Amir Sarid
\thanks{School of Mathematical Sciences, Raymond and Beverly Sackler Faculty of Exact Sciences, Tel Aviv University, Tel Aviv 6997801, Israel.
Email: {\tt amir@sarid.io}.}
}
\date{January 2022}

\usepackage[utf8]{inputenc}
\usepackage[hyphens]{url}
\usepackage[colorlinks]{hyperref}
\usepackage[hyphenbreaks]{breakurl}
\usepackage[backend=bibtex,style=numeric,maxnames=10,sortcites]{biblatex}
\usepackage{amssymb}
\usepackage{amsmath}
\usepackage{amsthm}
\usepackage{amsfonts}
\usepackage{dsfont}
\usepackage{nameref}
\usepackage{varioref}
\usepackage{cleveref}
\usepackage{tikz}
\usepackage{mathtools}
\usepackage{indentfirst}
\usepackage{placeins}

\numberwithin{equation}{section}
\numberwithin{figure}{section}
\theoremstyle{definition}
\newtheorem{theorem}{Theorem}[section]

\newtheorem{lemma}[theorem]{Lemma}
\newtheorem{claim}[theorem]{Claim}
\newtheorem{corollary}[theorem]{Corollary}
\newtheorem{proposition}[theorem]{Proposition}
\newtheorem{observation}[theorem]{Observation}

\DeclareMathOperator{\diam}{diam}
\DeclareMathOperator{\Tr}{Tr}
\DeclareMathOperator{\Var}{Var}
\DeclareMathOperator{\Cov}{Cov}

\begin{document}
\maketitle

\begin{abstract}
    We bound the second eigenvalue of random $d$-regular graphs, for a wide range of degrees $d$, using a novel approach based on Fourier analysis. Let $G_{n, d}$ be a uniform random $d$-regular graph on $n$ vertices, and $\lambda (G_{n, d})$ be its second largest eigenvalue by absolute value. For some constant $c > 0$ and any degree $d$ with $\log^{10} n \ll d \leq c n$, we show that $\lambda (G_{n, d}) = (2 + o(1)) \sqrt{d (n - d) / n}$ asymptotically almost surely. Combined with earlier results that cover the case of sparse random graphs, this fully determines the asymptotic value of $\lambda (G_{n, d})$ for all $d \leq c n$. To achieve this, we introduce new methods that use mechanisms from discrete Fourier analysis, and combine them with existing tools and estimates on $d$-regular random graphs --- especially those of Liebenau and Wormald \cite{liebenauwormald2017}.
\end{abstract}

\section{Introduction}

We consider the set of $d$-regular graphs on $n$ vertices, with $d = d(n)$ that may depend on $n$. It is well known that this set is nonempty if and only if $0 \leq d \leq n - 1$ and $nd$ is even. If that is indeed the case, we may ask about properties that most of the graphs in this set will have. In other words, if we define a random $d$-regular graph on $n$ vertices $G_{n,d}$ as a uniformly-chosen element of this set, we may inquire about the typical behavior of this $G_{n,d}$.

The typical behavior of the eigenvalues of $G_{n, d}$ is of particular interest. Define the adjacency matrix $A = (a_{ij})$ of a graph $G$ so that $a_{ij} = \mathds{1}_{\{i, j\} \in E(G)}$. We order the eigenvalues of $A$ so that $\lambda_1 \geq \lambda_2 \geq \dots \geq \lambda_n$. Since $G_{n,d}$ is always $d$-regular, all eigenvalues will be between $-d$ and $d$, with $\lambda_1 = d$.

The other eigenvalues may vary, and $\lambda_2, \lambda_n$ are especially important among them. The seminal work of Alon and Chung \cite{alonchung1988} shows a connection between the value of $\lambda = \max (|\lambda_2|, |\lambda_n|)$ and the edge distribution of the graph. We may ask, then --- what are the typical eigenvalues of the adjacency matrix of $G_{n,d}$? How are they distributed? What is the typical value of $\lambda$?

A lower bound on $\lambda$ in terms of $d$ and the diameter of the graph can be found by applying the Alon-Boppana Theorem (see \cite{alon1986, nilli1991}):
\begin{equation*}
    \lambda_2 \geq 2 \sqrt{d - 1} - \frac {2 \sqrt{d - 1} - 1} {\lfloor \diam(G) / 2 \rfloor } .
\end{equation*}
Note that $\diam (G) \geq \log_{d-1} n$, so it tends to infinity for $d$ sub-polynomial in $n$. Therefore, the theorem implies that $\lambda \geq (2 - o(1)) \sqrt{d - 1}$ for such $d$.

Also, note that the $d > n / 2$ case can be reduced to that of $d \leq n / 2$ by observing that $G_{n,d}$ and the complement of $G_{n, n-1-d}$ are identically distributed. It can be proven that whenever $\lambda_i \neq \lambda_1$ is an eigenvalue of $G_{n,d}$, $-1 - \lambda_i$ is an eigenvalue of its complement. The limiting distributions for $d$ and $n-d-1$ are therefore closely related. This explains why we may restrict our attention to $d \leq n / 2$, and we will freely assume this to be the case from now on.

Now, when considering the eigenvalues of $G_{n, d}$, we may hope for those eigenvalues to be distributed similarly to the better understood case of the \textit{binomial random graph} $G_{n,p}$, with each edge appearing with probability $p$ independently of all others. The $G_{n, p}$ analogy suggests that when $d \leq n / 2$ tends to infinity with $n$, the eigenvalues will be distributed according to Wigner's semicircle law \cite{wigner1958} --- and this was indeed proven by Tran and Vu \cite{tranvuwang2013}. That is, when scaling the eigenvalues by $(2 \sqrt{d (n - d) / n}) ^ {-1}$, their distributions will converge to a distribution with probability density function
\begin{equation*}
    f(x) = \left\{\begin{array}{lr}
        \frac {2} {\pi} \sqrt{1 - x^2}, & \text{if } |x| < 1\\
        0, & \text{otherwise}
        \end{array}\right. .
\end{equation*}

When $d$ is fixed, however, the limiting distribution is quite different, and its probability density function is given by McKay in \cite{mckay1981}:
\begin{equation*}
    f_d(x) = \left\{\begin{array}{lr}
        \frac {d \sqrt{4 (d-1) - x^2}} {2 \pi (d^2 - x^2)}, & \text{if } |x| < 2 \sqrt{d - 1}\\
        0, & \text{otherwise}
        \end{array}\right. .
\end{equation*}
Note that for large $d$, the $d^2 - x^2 \approx d^2$ for all $|x| \leq 2 \sqrt{d - 1}$. Therefore, the resulting $f_d(x)$ will be close to an appropriately scaled semicircle. Combined, those two results cover the entire range of possible values of $d$.

Note that since this is merely the limiting distribution, it only gives a one-sided bound on $\lambda$. Its absolute value cannot be too small, as that would imply the same for $\lambda_2, \lambda_3, \ldots, \lambda_n$, resulting in a significant deviation from the limiting distribution. It may still be possible, however, that $\lambda$ is often larger than $(2 + \varepsilon) \sqrt{d - 1}$ for some $\varepsilon > 0$. Indeed, the asymptotic magnitude of $\lambda$ is still not fully understood for all $d$.

A conjecture of Alon \cite{alon1986}, stating that for $d \geq 3$ fixed, $\lambda_2 \leq (2 + o(1)) \sqrt{d - 1}$ a.a.s., was proven by Friedman \cite{friedman2008} (see also \cite{bordenave2020, puder2015} for different approaches that achieve similar results), following pioneering work by Friedman, Kahn and Szemerédi \cite{friedman1989}, and by Broder, Frieze, Suen and Upfal \cite{bfsu1999}. The case of $d \leq n / 2$ that tends to infinity with $n$ was conjectured by Vu in \cite{vu2008} (also in \cite{vu2014}, Conjectures 7.3 and 7.4) to have
\begin{equation*}
    \lambda = (2 + o(1)) \sqrt{d (n - d) / n}
\end{equation*}
a.a.s.

The fact that $\lambda = \mathcal{O} (\sqrt{d})$ a.a.s.\ has been proven by Broder, Frieze, Suen and Upfal \cite{bfsu1999} for $d = o(\sqrt{n})$, and this result was extended for $d = \mathcal{O}(n ^ {2/3})$ by Cook, Goldstein and Johnson \cite{cook2018}. Recently, it was shown by Bauerschmidt, Huang, Knowles and Yau \cite{bhky2020} that for $d = o(n ^ {2/3})$ it indeed holds that $\lambda = (2 + o(1)) \sqrt{d}$ a.a.s.

In the $d = \Theta (n)$ case, Krivelevich, Sudakov and Vu \cite{ksv2007} show that $\lambda = \mathcal{O} (d ^ {7 / 8} \log ^ {1 / 8} d)$ a.a.s. Some bounds may also be inferred from concentration bounds --- for example, it follows from Alon, Krivelevich and Vu \cite{akv2002} that for $d = \Theta(n)$ the graph $G_{n, p}$ with $p = d /n$ is much more likely to be $d$-regular than to have $\lambda$ asymptotically larger than $\sqrt{d \log d}$. Therefore, $\lambda = O(\sqrt{d \log d})$. Finally, Tikhomirov and Youssef \cite{tikhomirovyoussef2019} prove that $\lambda = \mathcal{O}(\sqrt{d})$ a.a.s.\ for all $d \leq n / 2$.

\subsection{Our results}
\label{our_results_section}

Throughout the paper, $d = d(n)$ is chosen so that $d \leq n / 2$ and $nd$ is even. Also, $G_{n,d}$ is a random $d$-regular graph, $A$ is its adjacency matrix, $\lambda_1 = d \geq \lambda_2 \geq \dots \geq \lambda_n$ are the eigenvalues, and $\lambda = \max( |\lambda_2|, |\lambda_n| )$.

We introduce the notation $\mathbb{Z}^{n}_{\geq 0}$ for $n$-tuples of non-negative numbers, $p = d / (n-1)$ is the probability that a given edge will be in $G_{n,d}$, $[t] = \{1, 2, \ldots, t\}$, and $\binom {S} {t} = \{T \subseteq S : |T| = t\}$. We also use the notations $f(n) \ll g(n)$ and $g(n) \gg f(n)$, both of which are equivalent to $f(n) = o(g(n))$.

Our main result is:

\begin{theorem}
\label{main_theorem}

There exists a constant $c > 0$ such that for $1 \ll d \leq c n$, $G_{n,d}$ satisfies
\begin{equation*}
    \lambda = (2 + o(1)) \sqrt{ d (n - d) / n } = (2 + o(1)) \sqrt{n p (1 - p)}
\end{equation*}
a.a.s.

\end{theorem}

Small values of $d$ are dealt with in \cite{bhky2020}. In this paper we will only prove this for $d \gg \log ^ {10} n$, in which case we can employ the \textit{high trace method}:

Let $J$ be the all-ones matrix. It can be shown that $A - p J + p I$ has the same spectrum as $A$, except that $\lambda_1 = d$ has been replaced by $0$ and $\lambda_i$ has been replaced by $\lambda_i + p$ for $i > 1$. Thus, for any even integer $k > 0$, $(A - p J + p I)^k$ has eigenvalues $0, (\lambda_2 + p) ^ k, (\lambda_3 + p) ^ k, \ldots, (\lambda_n + p)^k$. Therefore, bounding $\Tr \left[ (A - p J + p I)^k \right] = (\lambda_2 + p)^k + (\lambda_3 + p)^k + \cdots + (\lambda_n + p)^k$ can be used to obtain an upper bound on $\lambda$.

Our plan is to directly apply the high trace method --- as used, for example, in \cite{furedikomlos1981} --- in order to bound the second eigenvalue of $d$-regular random graphs a.a.s. That is, we wish to bound $\mathds{E} \Tr \left [ (A - p J + p I)^k \right]$, where $k = k(n) \gg \log n$ is an even-valued function to be chosen by us. This turns out to require quite precise estimates for the probabilities of certain events on $G_{n, d}$.

Since $k$ is even, an upper bound $\mathds{E} \Tr \left[ (A - p J + p I)^k \right] \leq B^k$, for any $B = B(n)$, implies $\mathds{E} (\lambda + p) ^ k \leq B^k$. Applying the $k^{th}$ moment argument we see that $\lambda + p \leq (1+c) B$ a.a.s.\ for any $c > 0$, since $(1+c)^k$ grows to infinity with $n$. Therefore $\lambda + p \leq (1 + o(1)) B$ holds a.a.s., and for a suitable choice of $B$ this would prove \Cref{main_theorem}. It thus suffices to prove

\begin{theorem}
\label{actual_main_result}
There exists a constant $c > 0$ such that for any $\log n \ll k \ll d^{1 / 10}$, and any $d \leq c n$ with $n d$ even,
\begin{equation*}
    \mathds{E} \Tr \left [ (A - p J + p I)^k \right] \leq P(n) \left[ 4n p (1 -p) \right]^{k / 2}
\end{equation*}
for some polynomial term $P(n)$.
\end{theorem}

It should be noted that:
\begin{enumerate}
\item We assume $d \gg \log^{10} n$, so the range $\log n \ll k \ll d^{1 / 10}$ of possible $k$ values is nonempty. As explained above, smaller values of $d$ have already been dealt with elsewhere.
\item Since $k \gg \log n$, the $P(n)$ term in the bound will be negligible. Indeed, $(P(n)) ^ {1/k} = n^{o(1 / \log n)} = 1 + o(1)$.
\item Bounding $\lambda + p = (2 + o(1)) \sqrt{n p (1 - p)}$ is the same as bounding $\lambda = (2 + o(1)) \sqrt{n p (1 - p)}$, since $p \ll 2 \sqrt{n p (1 - p)}$ for $d \leq n / 2$.
\item This only establishes an upper bound on $\lambda $. This suffices, since the matching lower bound is a consequence of the limiting distribution of eigenvalues, as discussed above.
\end{enumerate}

Now, 
\begin{equation*}
    \left[ (A - p J + p I) ^ k \right]_{ij} = \sum\limits_{\Gamma} \prod\limits_{e \in \Gamma} (\mathds{1}_{e \in G} - p) ,
\end{equation*}
where the sum is taken over all non-lazy walks $\Gamma$ of length $k$ in the complete graph $K_n$ on $n$ vertices, going from $i$ to $j$. Since $\mathds{E} \mathds{1}_{e \in G} = p$ for any edge $e$, we may hope that such products will have small expected value. This requires us to obtain very precise estimates on $\mathds{P} [ S \subseteq G_{n,d} ]$ for various subsets $S \subseteq \Gamma$ --- at least as precise as we hope our eventual bound on the product would be.

Two obstacles may prevent this from always being the case:

First, this will not hold if the various edges are ``too correlated" with one another. Since $\mathds{E} \left[ (\mathds{1}_{e_1 \in G} - p) (\mathds{1}_{e_2 \in G} - p) \right] = \Cov (\mathds{1}_{e_1 \in G}, \mathds{1}_{e_2 \in G})$, if $e_1 \in G$ is highly correlated with $e_2 \in G$, we may get a large expected value. Although this is perhaps most obvious for a product of two terms, it indicates a problem that may arise with larger products as well. The second obstacle is that some edges may appear more than once in the product. For example, $\mathds{E} \left[ (\mathds{1}_{e \in G} - p)^2 \right] = \Var [ \mathds{1}_{e \in G} ] = p (1 - p)$. This may be viewed as a special case of the first obstacle, since $\mathds{1}_{e \in G}$ is trivially highly correlated with itself.

The standard way to deal with the first issue is to obtain some bound on those correlations, which would eventually yield accurate estimates for $\mathds{P} [ S \subseteq G_{n,d} ]$ for various subsets $S \subseteq \Gamma$. Those estimates must have very small margins of error, since they must eventually be translated into a bound on
\begin{equation*}
    \left| \prod\limits_{e \in \Gamma} (\mathds{1}_{e \in G} - p) \right|
\end{equation*}
for closed walks $\Gamma$, a bound whose order of magnitude will be $[p (1 - p)] ^ {k / 2}$ in most cases.

The second issue is often handled by separating the product into repeating edges and non-repeating edges. For example, in \cite{alonnussbaum2008}, the correlations issue is subverted by assuming the edges to be $k$-wise independent, which causes the expected value of the product to be $0$ unless all edges are repeating. The case of closed walks with all edges repeating, however, still requires careful treatment, which is done using precise estimates on the number of such walks.

The structure of this paper is as follows:

\Cref{prelim} of this paper lays the groundwork for later sections: some combinatorial bounds, facts from Fourier analysis, and claims about $d$-regular graphs are established in this section. In that section, we also obtain some preliminary bounds on the probability of small subsets of edges to all appear (or to not appear) simultaneously in $G_{n, d}$.

\Cref{section_contribution} is devoted to using those bounds for proving the fact that small enough subsets of the variables $\mathds{1}_{e \in G}$ are generally not strongly correlated, so that the first obstacle listed above will be less of an issue. This leads to a highly accurate estimate on $\mathds{P} [ S \subseteq G_{n,d} ]$ for various subsets $S \subseteq \Gamma$, and eventually to a bound on the expected products for each individual walk $\Gamma$. The resulting bound depends on the number of edges appearing more than once in $\Gamma$ and the number of distinct vertices through which it passes.

Finally, \Cref{section_enumeration} bounds the number of walks $\Gamma$ of each type. This allows us to sum the previously obtained bound over all possible closed walks $\Gamma$, to obtain a bound on $\mathds{E} \Tr \left[(A - p J + p I)^k \right]$, and to prove \Cref{actual_main_result}. \Cref{main_theorem} then follows, as explained above.

\section{Preliminaries}
\label{prelim}

\subsection{Combinatorial Bounds}

This subsection deduces bounds on certain combinatorial sums, in order to be readily used when needed.

\begin{lemma}
\label{fourier_product_lemma_1}
For some large enough constant $C > 0$ and some small enough constant $c > 0$, and for any $0 \leq m \leq n$ and $a \leq c n^{-5}$,
\begin{equation*}
    \sum\limits_{k=0}^{n} \binom {n} {k} (4(m+k))! a^{m+k} \leq C (4m)! a^m .
\end{equation*}

\end{lemma}

\begin{proof}
Denote the $k^{th}$ term in the sum by $T_k$. Then $T_{k+1} / T_k \leq n (8n)^4 a \leq 4096 n^5 a \leq 4096 c$. Thus, for $c \leq 0.5 / 4096$, the sum will be at most $2 T_0 = 2 (4m)! a^m$.
\end{proof}

\begin{lemma}
\label{fourier_inverse_lemma_2}
For some absolute constant $C > 0$,
\begin{equation*}
    \sum\limits_{k=1}^{n} \sum\limits_{\substack{s_1 + s_2 + \cdots + s_k = n \\ s_1, \ldots, s_k > 0}} \binom{n} {s_1, \ldots, s_k} \prod\limits_{i=1}^k (4s_i)! \leq C (4n)! .
\end{equation*}
\end{lemma}

\begin{proof}
We begin by fixing some $1 \leq k \leq n$.

Each term of the sum has at least one $i$ such that $s_i$ is maximized. By symmetry, this sum is therefore at most $k$ times the contribution of terms where the maximum is achieved by $s_1$. Thus, it suffices to only consider terms with $s_1 = \max( s_1, \ldots, s_k)$, and to show that their sum is dominated by the term $s_1 = n - k + 1, s_2 = s_3 = \cdots = s_k = 1$, and that all other terms combined are $o(1)$ times the contribution of this term. Indeed, this main term is then precisely
\begin{equation*}
    \binom {n} {n-k+1, 1, 1, \ldots, 1} (4n - 4k + 4)! \cdot 4! \cdot 4! \cdots 4! = \frac {n!} {(n - k + 1)!} (4n - 4k + 4)! 24^{k-1},
\end{equation*}
times $k$ since the maximum can actually be obtained by any one of $s_1, s_2, \ldots, s_k$. Since this product rapidly decreases with $k$, summing this over all $k$ will result in some absolute constant $C_1$ times the $k = 1$ term, which is precisely $(4n)!$.

Any other term can be obtained from the term corresponding to $s_1 = n - k + 1, s_2 = 1, \ldots, s_k = 1$ by repeatedly removing $1$ from $s_1$ and adding it to some $s_i$. Since $s_1$ remains maximal, this process may only change the summand by a factor of
\begin{equation*}
    \frac {s_1} {s_i + 1} \cdot \frac {(4 s_i + 1) (4 s_i + 2) (4 s_i + 3) (4 s_i + 4)} {(4 s_1 - 3) (4 s_1 - 2) (4 s_1 - 1) 4 s_1} = \frac {(4 s_i + 1) (4 s_i + 2) (4 s_i + 3)} {(4 s_1 - 3) (4 s_1 - 2) (4 s_1 - 1)} < 1.
\end{equation*}
Furthermore, if this transition increased $s_i$ from $1$ to $2$ while keeping $s_1 \geq n / 2$, then this factor is at most $C_2 n^{-3}$.

Denote $I = \{i \in [k] : s_i > 1 \}$ and $k_1 = |I|$. We first transfer $1$ from $s_1$ to each $s_i$ with $i \in I$, multiplying the term by $C_2 n^{-3}$ for the first $\min (n/2, k_1)$ times. Any subsequent transfer from $s_1$ to some $s_i$ will not increase the term, and therefore any term with a given $k_1$ value is at most $(C n^{-3}) ^ {\min(n/2, k_1)}$ times the first term, for some $C > 0$.

Now, for any fixed $k$ and $k_1$, there are $\binom {k} {k_1}$ ways to choose the subset $I$, and $\binom {n-(k+k_1)+k_1-1} {k_1 - 1}$ ways to choose the values $s_i$ for all $i \in I$. Thus, the number of terms with a given $k_1$ is $\binom {k} {k_1} \binom {n-k-1} {k_1 - 1}$. This term grows by a factor of at most $C_3 n^2$ when $k_1 < n/2$ is increased by $1$, and does not increase at all when $k_1 \geq n/2$ is increased by $1$.

With $k$ fixed, increasing $k_1$ by $1$ multiplies the total contribution of terms with this $k_1$ by $C_2 C_3 n^{-1}$ for the first $n/2$ values of $k_1$, and does not increase it for larger $k_1$. Summing this over all $k_1$, we see that the total is at most some constant $C_4$ times the $k_1 = 1$ term, which is just
\begin{equation*}
    \frac {n!} {(n - k + 1)!} (4n - 4k + 4)! 24^{k-1} . \qedhere
\end{equation*}
\end{proof}

\subsection{Properties of Fourier Coefficients}
\label{fourier_props}

In order to obtain accurate estimates for certain quantities on $d$-regular random graphs, we use the theory of functions $f(x_1, x_2, \ldots, x_t)$ from the cube $\{0,1\}^t$ to $\mathbb{R}$. We employ the following shorthand: given $f \colon \{0, 1\}^t \rightarrow \mathbb{R}$, and $T \subseteq [t]$, we freely write $f(T)$ instead of $f(\mathds{1}_T)$, and $x_T$ instead of $\prod_{i \in T} x_i$.

Specifically, we are motivated by the following: given a sequence of edges $e_1, e_2, \ldots, e_t \in K_n$, and for any $1 \leq i \leq t$, define $p_{i} \colon \{0,1\}^{i-1} \rightarrow \mathbb{R}$ by
\begin{equation*}
    p_{i} (x_1, \ldots, x_{i-1}) = \mathds{P} [ e_i \in G_{n, d} \mid \bigcap\limits_{\substack{1 \leq j < i \\ x_j = 1}} \{ e_j \in G_{n, d} \} \cap \bigcap\limits_{\substack{1 \leq j < i \\ x_j = 0}} \{ e_j \notin G_{n, d} \} ] .
\end{equation*}

Then, for any sequence of edges $e_1, \ldots, e_t$, we will eventually obtain a bound on the expectation
\begin{equation*}
    \mathds{E} \left[ \prod\limits_{e \in \Gamma} (\mathds{1}_{e \in G} - p) \right],
\end{equation*}
using the Fourier coefficients of the functions $p_{1}, \ldots, p_{t}$.

We reference \cite{odonnell2014} throughout this section for standard results about the Fourier transform of functions on $\{0, 1\}^t$. Note that the notation there uses functions from $\{-1, 1\}^t$ instead, so some conversion is needed.

Any $f \colon \{0, 1\}^t \rightarrow \mathbb{R}$ admits a unique representation as a polynomial $f(x) = \sum_{T \subseteq [t]} a_T x_T$ for some choice of coefficients $a_T$ (see \cite{odonnell2014}, Theorem 1.1). We denote those coefficients $a_T$ corresponding to a specific $f$ by $\chi_T (f)$, and refer to them as the \textit{Fourier coefficients} of $f$. As proven in \cite{odonnell2014}, Proposition 1.8, they are given by
\begin{equation}
\label{fourier_formula}
    \chi_T (f) = \sum\limits_{S \subseteq T} (-1)^{|T| + |S|} f(S) .
\end{equation}

A formula for the Fourier transform of a product $f \cdot g$ can now be obtained from Theorem 1.27 and Exercise 1.12 of \cite{odonnell2014},
\begin{equation}
\label{fourierprod}
    \chi_S (f g) = \sum\limits_{S_1 \cup S_2 = S} \chi_{S_1} (f) \chi_{S_2} (g) .
\end{equation}

The following is a consequence of this:

\begin{corollary}
\label{fourierdiv}
If $f(0, \ldots, 0) = 1$ and $f(x_1, \ldots, x_t)$ is bounded away from $0$ for any $(x_1, \ldots, x_t) \in \{0, 1\}^t$, then
\begin{equation*}
    \chi_S (1 / f) = \sum\limits_{i=0}^{\infty} \sum\limits_{\substack {S_1 \cup \cdots \cup S_i = S \\ S_1, \ldots, S_j \neq \emptyset}} \  \prod\limits_{j=1}^{i} (-\chi_{S_j}(f)).
\end{equation*}
\end{corollary}

\begin{proof}
Since $f$ is bounded away from $0$, we have
\begin{equation}
\label{expanded_inverse_fourier}
    \frac {1} {f} = \frac {1} {1 - (1 - f)} = \sum\limits_{i=0}^{\infty} (1-f)^i .
\end{equation}

Using \Cref{fourierprod} repeatedly, we see that
\begin{equation*}
    \chi_S((1-f)^i) = \sum\limits_{S_1 \cup \cdots \cup S_i = S} \prod\limits_{j=1}^i \chi_{S_j} (1 - f)
\end{equation*}
for all $i$. Since $f(0, \ldots, 0) = 1$, we will have $\chi_{T} (1-f) = 0$ if $T = \emptyset$, and otherwise $\chi_{T} (1-f)= -\chi_{T} (f)$. Plugging this into \Cref{expanded_inverse_fourier} proves the corollary.
\end{proof}

\subsection{Bounds of Fourier Coefficients}
\label{fourier_bounds}

Throughout this section, we assume that the functions $f, g$ have $f(0, \ldots, 0) = g(0, \ldots, 0) = 1$. We can now use the properties of the Fourier transform from \Cref{fourier_props} to show how bounds on the Fourier coefficients of the functions $f, g$ can be made into bounds on the Fourier coefficients of $f \cdot g$ and $1 / f$.

\begin{proposition}
\label{prod_bound}

For some large enough constant $C > 0$ and some small enough constant $c > 0$, assume that $a \leq c t ^ {-5}, b \leq 1$ and that $|\chi_{S} (f)|, |\chi_{S} (g)| \leq (4|S|)! a^{|S|} b$, for all nonempty $S \subseteq [t]$. Then $|\chi_{S} (f \cdot g)| \leq C (4|S|)! a^{|S|} b$, for all nonempty $S \subseteq [t]$.

\end{proposition}

\begin{proof}
Since we are only concerned with subsets of $S$, we may assume $S = [t]$. Now, applying \Cref{fourierprod}, we see that
\begin{equation*}
    |\chi_{S} (f \cdot g)| \leq \sum\limits_{S_1 \cup S_2 = [t]} |\chi_{S_1} (f)| \cdot |\chi_{S_2} (g)| = \sum\limits_{S_1 \subseteq [t]} |\chi_{S_1} (f)| \sum\limits_{ \substack{ S_2 \subseteq [t] \\ S_1 \cup S_2 = [t]} } |\chi_{S_2} (g)|.
\end{equation*}

We partition the inner sum by the size $k$ of $S_1 \cap S_2$. Given any $S_1 \subseteq [t]$, the number of options for such $S_2$ of size $t - |S_1| + k$ is $\binom {|S_1|} {k}$, and thus the inner sum is bounded by
\begin{equation*}
    \sum\limits_{k = 0}^{|S_1|} \binom {|S_1|} {k} (4k+4t-4|S_1|)! a ^ { k+t-|S_1| } ,
\end{equation*}
times a factor of $b \leq 1$ if $S_1 \neq [t]$. By \Cref{fourier_product_lemma_1}, this sum is at most $C_1 (4t - 4|S_1|)! a^{t - |S_1|}$. Multiplying by $|\chi_{S_1} (f)| \leq (4 |S_1|)! a^{|S_1|}$, times an additional $b \leq 1$ factor if $S_1 \neq \emptyset$, we see that the sum of all terms involving a specific $S_1$ is at most $C_1 (4|S_1|)! (4t - 4|S_1|)! a^{t} b$.

Denote by $\Phi_k$ the sum of terms with $|S_1| = k$. We thus have $\Phi_k \leq C_1 \binom {t} {k} (4k)! (4t-4k)! a^{t} b$. Since the sum of $\binom {t} {k} (4k)! (4t-4k)!$ is dominated by the terms at $k=0$ and $k=t$, the sum will be dominated by $\Phi_0 + \Phi_n$, which is $C_2 (4t)! a^{t} b$ for some absolute constant $C_2$.
\end{proof}

\begin{proposition}
\label{inv_bound}
For some large enough constant $C > 0$ and some small enough constant $c > 0$, assume $f(x_1, \ldots, x_t)=1+o(t^{-1})$ for all $x_1, \ldots, x_t \in \{0, 1\}^t$, that $a \leq c t ^ {-6}, b \leq 1$, and that $|\chi_{S} (f)| \leq (4|S|)! a^{|S|} b$, for all nonempty $S \subseteq [t]$. Then $|\chi_{S} (1 / f)| \leq C (4|S|)! a^{|S|} b$, for all nonempty $S \subseteq [t]$.
\end{proposition}

\begin{proof}
We may assume $S = [t]$ once again. Using \Cref{fourierdiv}, we have
\begin{equation*}
    |\chi_{S} (1 / f)| \leq \sum\limits_{k=1}^{\infty} \sum\limits_{\substack {S_1 \cup \cdots \cup S_k = S \\ S_1, \ldots, S_k \neq \emptyset}} \prod\limits_{i=1}^k |\chi_{S_i} (f)| .
\end{equation*}

Fix any value of $k$. Given any $S_1, S_2, \ldots, S_{k-1}$ so that $|S_1 \cup S_2 \cup \cdots \cup S_{k-1}| = t_1 < t$, we partition the sum over $S_k$ according to $t_2$, the size of $(S_1 \cup S_2 \cup \cdots \cup S_{k-1}) \cap S_k$. The number of options for $S_k$ of size $t - t_1 + t_2$ is $\binom {t_1} {t_2}$, and each term corresponding to such $S_k$ is bounded by
\begin{equation*}
    (4t - 4t_1 + 4t_2)! a^{t - t_1 + t_2} b \prod\limits_{i=1}^{k-1} (4|S_i|)! a^{|S_i|} b .
\end{equation*}

Thus, summing over all possible options for $S_k$, the bound becomes
\begin{equation*}
    \prod\limits_{i=1}^{k-1} (4|S_i|)! a^{|S_i|} b \sum\limits_{t_2=0}^{t_1} \binom {t_1} {t_2} (4t - 4t_1 + 4t_2)! a^{t - t_1 + t_2}  b.
\end{equation*}
Since each term of the sum is only a $C_1 t^5 a$ fraction of the previous one, the sum is dominated by the $t_2 = 0$ term. Therefore, the sum is at most $(1 + C_2 t^{-1}) (4t - 4t_1)! a^{t - t_1}$.

If, however, $t_1 = t$, then since $S_k$ cannot be empty the bound is now
\begin{equation*}
    \prod\limits_{i=1}^{k-1} (4|S_i|)! a^{|S_i|} b \sum\limits_{t_2=1}^{t_1} (4t_2)! a^{t_2} b ,
\end{equation*}
and the sum is now at most $(24 + C_3 t^{-1}) a$ by the same reasoning as above. Choosing $c$ small enough, we may force $C_2, C_3$ to be arbitrarily small.

The same reasoning now holds for summing over all options $S_{k-1}$ that will yield a union of size $|S_1 \cup S_2 \cup \cdots \cup S_{k-1}| = t_1$ given the previous size $|S_1 \cup S_2 \cup \cdots \cup S_{k-2}| = \Tilde{t_1} < t_1$, we see that the sum is once again dominated by the term with $|S_{k-1}| = t_1 - \Tilde{t_1}$, or with $|S_{k-1}| = 1$ if $t_1 = \Tilde{t_1}$.

Continuing similarly for $S_{k-2}, S_{k-3}$ and so on, after $k$ iterations we eventually can bound the sum total of all the terms with any given $s_i = |S_i \setminus (S_1 \cup S_2 \cup \cdots \cup S_{i-1})|$. Specifically, if we denote $k_0 = |\{s_i = 0\}|, k_1 = |\{s_i > 0\}|$, this will yield a bound of
\begin{equation}
\label{sum_given_s}
    C_4 a^t (24 a)^{k_0} b^{k_0 + k_1} \prod\limits_{i=1}^{k_1} (4s_i)!
\end{equation}
for $C_2, C_3$ small enough. Note that $b^{k_0 + k_1} = b^t \leq b$, since $b \leq 1$. Therefore, summing over all non-negative $s_1, s_2, \ldots, s_k$ such that $s_1 + \cdots + s_k = t$, we get at most
\begin{equation*}
    C_4 a^t b \sum\limits_{s_1+s_2+\cdots+s_k=t} \binom {t} {s_1, \ldots, s_k} (24 a)^{k_0} \prod\limits_{i=1}^{k_1} (4s_i)!.
\end{equation*}

Instead of summing \Cref{sum_given_s} over all values of $k$, we sum over $k_0$, $k_1$: for each possible value, there are $\binom{k_1 + k_0} {k_0} \leq t^{k_0}$ options to choose which $s_i$ will be zero, so the sum is at most
\begin{equation*}
    C_4 a^t b \left[ \sum\limits_{k_0=0}^{\infty} (24 t a)^{k_0} \right] \left[ \sum\limits_{k_1=1}^{t}  \sum\limits_{\substack {s_1 + s_2 + \cdots +s_{k_1} = t \\ s_1, \ldots, s_{k_1} > 0}} \binom {t} {s_1, \ldots, s_{k_1}} \prod\limits_{i=1}^{k_1} (4s_i)! \right] .
\end{equation*}
The sum over $k_0$ is $1 + C_5 t a \leq C_6$. Applying \Cref{fourier_inverse_lemma_2}, we conclude that the sum over $k_1$ is at most some constant $C_7$ times $(4t)!$. The resulting bound is therefore $C_8 (4t)! a^t b$, for some constant $C_8$.
\end{proof}

\subsection{Properties of Random Regular Graphs}

This subsection presents some basic results regarding $d$-regular graphs. Recall that we assume $nd$ to be even.

\begin{proposition}
\label{graph_existence}
Let $\mathcal{A},\mathcal{B} \subseteq \binom {[n]} {2}$ be disjoint sets with $|\mathcal{A} \cup \mathcal{B}| = k$, with $2 k \leq d$ and $4 k \leq n - d - 1$. Then there exists a $d$-regular graph $G$ on $n$ vertices, with $\mathcal{A} \subseteq E(G)$ and $\mathcal{B} \cap E(G) = \emptyset$.
\end{proposition}

\begin{proof}
The union $\mathcal{A} \cup \mathcal{B}$ touches at most $2k \leq d$ vertices. Denote a set of $d + 1$ vertices containing all of those by $V_1$, and let $V_2 = [n] \setminus V_1$, and construct a clique $G_1$ on $V_1$ and an arbitrary $d$-regular graph $G_2$ on $V_2$. Their disjoint union results in a graph containing all edges from $\mathcal{A} \cup \mathcal{B}$.

Now, each edge $b = \{v_1, v_2\} \in \mathcal{B}$ needs to be removed, while keeping the graph $d$-regular. This can be done by preforming a \textit{switching} on $b$, together with some edge $e = \{v_3, v_4\} \in G_2$. This is done by removing those edges and adding the two edges $\{v_1, v_3\}, \{v_2, v_4\}$ in their place (see \Cref{forward_switching}). In order to do this for all of $\mathcal{B}$ simultaneously, the respective edges $e$ found in $G_2$ should form a matching. We would also like for them not to be incident to any edge in $\mathcal{B}$.

Thus, all that remains is to find a matching of size $|\mathcal{B}| \leq k$ in $V_2$. Indeed, $G_2$ is a $d$-regular graph on $n - d - 1$ vertices, so Vizing's theorem implies that the chromatic index $\chi ' (G_2) \leq d + 1$. The largest color class constitutes a matching of size at least
\begin{equation*}
    \frac {(n - d - 1) d} {2 (d + 1)} \geq \frac {n - d - 1} {4} \geq k . \qedhere
\end{equation*}
\end{proof}

\FloatBarrier
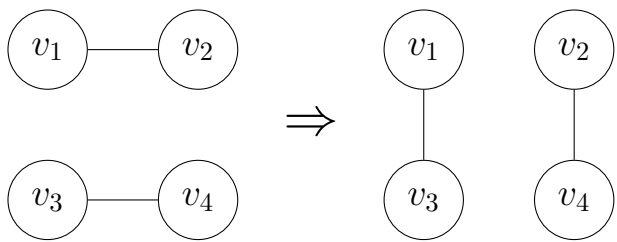
\begin{figure}
\centering
\begin{tikzpicture}[font={\fontsize{14pt}{12}\selectfont}]
\node[circle,draw=black,minimum size=30pt] (a00) at (0, 2) {$v_1$};
\node[circle,draw=black,minimum size=30pt] (a01) at (2, 2) {$v_2$};
\node[circle,draw=black,minimum size=30pt] (a10) at (0, 0) {$v_3$};
\node[circle,draw=black,minimum size=30pt] (a11) at (2, 0) {$v_4$};
\node[circle,draw=black,minimum size=30pt] (b00) at (5, 2) {$v_1$};
\node[circle,draw=black,minimum size=30pt] (b01) at (7, 2) {$v_2$};
\node[circle,draw=black,minimum size=30pt] (b10) at (5, 0) {$v_3$};
\node[circle,draw=black,minimum size=30pt] (b11) at (7, 0) {$v_4$};
\draw (a00) -- (a01);
\draw (a10) -- (a11);
\node[align=left, font={\fontsize{20pt}{12}\selectfont}] at (3.5,1) {$\Rightarrow$};
\draw (b00) -- (b10);
\draw (b01) -- (b11);
\end{tikzpicture}
\caption{The forward switching operation}
\label{forward_switching}
\end{figure}

We will later require precise estimates for certain quantities of random graphs which are nearly, but not exactly, $d$-regular. To that end, the following notation is used: given a set $\mathcal{A} \subseteq \binom {[n]} 2$ of allowed edges and a vector of degrees $\mathbf{d} \in \mathbb{Z}_{\geq 0}^n$, we denote by $\mathcal{G}_{\mathbf{d}, \mathcal{A}}$ the set of all graphs with degree sequence $\mathbf{d}$ and $E(G) \subseteq \mathcal{A}$. We similarly denote by $G_{\mathbf{d}, \mathcal{A}}$ a uniformly chosen element of this set. Therefore, $G_{n, d}$ is just $G_{\mathbf{d}, \mathcal{A}}$ for $\mathbf{d} = (d, d, \ldots, d)$ and $\mathcal{A} = \binom {[n]} {2}$.

To allow for small changes to the degree sequence, we denote by $\mathbf{e}_i \in \mathbb{Z}^n$ the vector with $1$ in the $i$th entry and $0$ anywhere else. Also, when subtracting elements of $\mathbb{Z}^n$, the subtraction is done coordinate-wise --- for example, $\mathbf{d} - \mathbf{e}_a - \mathbf{e}_b$ is just $\mathbf{d}$ with $1$ removed from the entries for $a$ and $b$.

We may now consider various properties of this set of graphs. Specifically, we introduce the notation:
\begin{enumerate}
    \item $N_{\mathbf{d}, \mathcal{A}} = |\mathcal{G}_{\mathbf{d}, \mathcal{A}}|$,
    \item $P_{\mathbf{d}, \mathcal{A}} (ab) = \mathds{P} [\{a,b\} \in E(G_{\mathbf{d}, \mathcal{A}})|$, and
    \item $Y_{\mathbf{d}, \mathcal{A}} (abc) = \mathds{P} [\{a,b\},\{b,c\} \in E(G_{\mathbf{d}, \mathcal{A}})|$.
\end{enumerate}
Our notation is similar to that of \cite{liebenauwormald2017}, but not completely consistent with it. We also denote $\mathcal{A}(v) = \{u \in [n] : \{v, u \} \in \mathcal{A} \}$, the ``projection'' of $\mathcal{A}$ to $v$. With $\mathbf{d}$ implicitly given, we also have $\mathcal{A}^*(v) = \{u \in [n] : \{v, u \} \in \mathcal{A}, P_{\mathbf{d}, \mathcal{A}} (ab) > 0 \}$.

We can now use this notation to state the following useful fact:
\begin{observation}
\label{conditioning_smaller_graph}
When conditioning on the existence of $\{a, b\} \in E(G_{\mathbf{d}, \mathcal{A}})$, the rest of the graph behaves just like $G_{\mathbf{d} - \mathbf{e}_a - \mathbf{e}_b, \mathcal{A} \setminus \{a, b\}}$ (assuming $P_{\mathbf{d}, \mathcal{A}} (ab) > 0$ so that such conditioning is possible). Similarly, whenever it is possible to condition on $\{a, b\} \notin E(G_{\mathbf{d}, \mathcal{A}})$, the rest of the graph behaves like $G_{\mathbf{d}, A \setminus \{a, b\}}$.

This implies, for example, that $Y_{\mathbf{d}, \mathcal{A}} (abc) = P_{\mathbf{d}, \mathcal{A}} (ab) P_{\mathbf{d} - \mathbf{e}_a - \mathbf{e}_b, \mathcal{A} \setminus \{a, b \}} (bc)$. Similarly, any term $\mathds{P} [\mathcal{B} \subseteq E(G_{\mathbf{d}, \mathcal{A}}), \mathcal{C} \cap E(G_{\mathbf{d}, \mathcal{A}}) = \emptyset]$ can be expanded as a product of $P$ and $1 - P$ terms with modified degree sequence and allowed edges.
\end{observation}

The following recursive formulae for $P_{\mathbf{d}, \mathcal{A}} (ab), Y_{\mathbf{d}, \mathcal{A}} (abc)$ are proven by Liebenau and Wormald in \cite{liebenauwormald2017}:

\begin{proposition}
\label{recursive_estimates}

If $P_{\mathbf{d}, \mathcal{A}} (ab) > 0$, then
\begin{equation}
\label{recursive_p}
P_{\mathbf{d}, \mathcal{A}} (ab) = \mathbf{d} (b) \left( \sum\limits_{c \in \mathcal{A}^*(b)} \frac {\mathbf{d}(c) - B_{\mathbf{d} - \mathbf{e}_a - \mathbf{e}_b, \mathcal{A}} (ca)} {\mathbf{d}(a) - B_{\mathbf{d} - \mathbf{e}_b - \mathbf{e}_c, \mathcal{A}} (ac)} \cdot \frac {1 - P_{\mathbf{d} - \mathbf{e}_b - \mathbf{e}_c, \mathcal{A}}(bc)} {1 - P_{\mathbf{d} - \mathbf{e}_a - \mathbf{e}_b, \mathcal{A}} (ab)} \right)^{-1} ,
\end{equation}
where $B_{\mathbf{d}, \mathcal{A}} (ab)$ is defined by
\begin{equation}
\label{recursive_b}
B_{\mathbf{d}, \mathcal{A}} (ab) = \sum\limits_{c \in \mathcal{A}(a) \setminus \mathcal{A}(b)} P_{\mathbf{d}, \mathcal{A}} (ac) + \sum\limits_{c \in \mathcal{A}(a) \cap \mathcal{A}(b)} Y_{\mathbf{d}, \mathcal{A}} (acb) .
\end{equation}
This quantity can be interpreted as the expected number of edges incident to $a$ that cannot be switched to be incident to $b$. Also,
\begin{equation}
\label{recursive_y}
Y_{\mathbf{d}, \mathcal{A}} (abc) = P_{\mathbf{d}, \mathcal{A}} (ab) \frac {P_{\mathbf{d} - \mathbf{e}_a - \mathbf{e}_b, \mathcal{A}} (bc)-Y_{\mathbf{d} - \mathbf{e}_a - \mathbf{e}_b, \mathcal{A}} (abc)} {1 - P_{\mathbf{d} - \mathbf{e}_a - \mathbf{e}_b, \mathcal{A}} (ab)} .
\end{equation}
\end{proposition}

\begin{proof}
\Cref{recursive_p} stems from the fact that $b$ has precisely $\mathbf{d}(b)$ neighbors in every $G \in \mathcal{G}_{\mathbf{d}, \mathcal{A}}$. Averaging over all such $G$, we get $\mathbf{d}(b) = \sum\limits_{c \in \mathcal{A}^*(b)} P_{\mathbf{d}, \mathcal{A}} (bc)$. Thus, we have
\begin{equation}
\label{p_ratios_sum}
    \frac {\mathbf{d}(b)} {P_{\mathbf{d}, \mathcal{A}} (ab)} = \sum\limits_{c \in \mathcal{A}^*(b)} \frac {P_{\mathbf{d}, \mathcal{A}} (bc)} {P_{\mathbf{d}, \mathcal{A}} (ab)} .
\end{equation}

In order to calculate the ratios $P_{\mathbf{d}, \mathcal{A}} (bc) / P_{\mathbf{d}, \mathcal{A}} (ab)$, note that $(1 - P_{\mathbf{d} - \mathbf{e}_a - \mathbf{e}_b, \mathcal{A}} (ab)) N_{\mathbf{d} - \mathbf{e}_a - \mathbf{e}_b, \mathcal{A}} = N_{\mathbf{d} - \mathbf{e}_a - \mathbf{e}_b, \mathcal{A} \setminus \{a, b\}} = P_{\mathbf{d}, \mathcal{A}} (ab) N_{\mathbf{d}, \mathcal{A}}$, as implied by \Cref{conditioning_smaller_graph}. A similar equality holds for graphs with degree sequence $\mathbf{d} - \mathbf{e}_b - \mathbf{e}_c$, and therefore
\begin{equation}
\label{p_ratio_to_n_ratio}
    \frac {P_{\mathbf{d}, \mathcal{A}} (bc)} {P_{\mathbf{d}, \mathcal{A}} (ab)} = \frac {P_{\mathbf{d}, \mathcal{A}} (bc) N_{\mathbf{d}, \mathcal{A}}} {P_{\mathbf{d}, \mathcal{A}} (ab) N_{\mathbf{d}, \mathcal{A}}} = \frac {1 - P_{\mathbf{d} - \mathbf{e}_b - \mathbf{e}_c, \mathcal{A}} (bc)} {1 - P_{\mathbf{d} - \mathbf{e}_a - \mathbf{e}_b, \mathcal{A}} (ab)} \cdot \frac {N_{\mathbf{d} - \mathbf{e}_b - \mathbf{e}_c, \mathcal{A}}} {N_{\mathbf{d} - \mathbf{e}_a - \mathbf{e}_b, \mathcal{A}}} .
\end{equation}

On the other hand, the ratio between the number of graphs with degree sequence $\mathbf{d} - \mathbf{e}_a - \mathbf{e}_b$ and graphs with degree sequence $\mathbf{d} - \mathbf{e}_b - \mathbf{e}_c$ can be determined by a double counting argument: one can count the average number of ways to switch forwards and backwards between such graphs. Together with \Cref{p_ratio_to_n_ratio}, this turns out to yield the formula
\begin{equation*}
    \frac {P_{\mathbf{d}, \mathcal{A}} (bc)} {P_{\mathbf{d}, \mathcal{A}} (ab)} = \frac {1 - P_{\mathbf{d} - \mathbf{e}_b - \mathbf{e}_c, \mathcal{A}} (bc)} {1 - P_{\mathbf{d} - \mathbf{e}_a - \mathbf{e}_b, \mathcal{A}} (ab)} \cdot \frac {N_{\mathbf{d} - \mathbf{e}_b - \mathbf{e}_c, \mathcal{A}}} {N_{\mathbf{d} - \mathbf{e}_a - \mathbf{e}_b, \mathcal{A}}} = \frac {\mathbf{d}(c) - B_{\mathbf{d} - \mathbf{e}_a - \mathbf{e}_b}(ca)} {\mathbf{d}(a) - B_{\mathbf{d} - \mathbf{e}_b - \mathbf{e}_c}(ac)} .
\end{equation*}

\Cref{recursive_p} now follows by plugging this back into the sum of \Cref{p_ratios_sum}, and solving for $P_{\mathbf{d}, \mathcal{A}} (ab)$.

\Cref{recursive_y} is significantly easier to prove, and ultimately amounts to expanding $Y_{\mathbf{d}, \mathcal{A}} (abc)$ as $P_{\mathbf{d}, \mathcal{A}} (ab) P_{\mathbf{d} - \mathbf{e}_a - \mathbf{e}_b, \mathcal{A} \setminus \{a,b\}} (bc)$, as explained in \Cref{conditioning_smaller_graph}. By conditioning on the fact that $\{a,b\} \notin G_{\mathbf{d} - \mathbf{e}_a - \mathbf{e}_b, \mathcal{A}}$, we obtain
\begin{equation*}
    P_{\mathbf{d} - \mathbf{e}_a - \mathbf{e}_b, \mathcal{A} \setminus \{a,b\}} (bc) = \frac {P_{\mathbf{d} - \mathbf{e}_a - \mathbf{e}_b, \mathcal{A}} (bc)-Y_{\mathbf{d} - \mathbf{e}_a - \mathbf{e}_b, \mathcal{A}} (abc)} {1 - P_{\mathbf{d} - \mathbf{e}_a - \mathbf{e}_b, \mathcal{A}} (ab)} . \qedhere
\end{equation*}
\end{proof}

Repeated use of \Cref{recursive_estimates} will yield gradually improving estimates for $P_{\mathbf{d}, \mathcal{A}}$, but we need to plug some initial estimate. Note that those estimates must apply not only to $P_{\mathbf{d}, \mathcal{A}}$ but to slightly modified degree sequences $\mathbf{d}'$. For brevity, for any integer $t \geq 0$, we introduce the notation
\begin{equation*}
    B_t (n, d) = \{\mathbf{d} \in [d]^n : \sum\limits_{i=1}^n \mathbf{d}(i) \geq dn - 2t \} .
\end{equation*}

\begin{claim}
\label{first_estimate}
Let $d \leq c n$ for some small enough $c$, $t \ll d$, $\mathcal{A} \subseteq \binom {[n]} {2}$ missing at most $t$ edges, and $\mathbf{d} \in B_t(n, d)$. Then $P_{\mathbf{d}, \mathcal{A}} (ab) \leq (1 + 4c) d / n$.
\end{claim}

\begin{proof}
This is very similar to Lemma 2.3 in \cite{liebenauwormald2017}, with the slight modification that $\mathcal{A}$ may be slightly smaller than $\binom {[n]} {2}$.

Recall the switching operation, as described in \Cref{graph_existence} and illustrated in \Cref{forward_switching}. This operation takes two edges $\{ a, b \}, \{c, d\} \in G$ such that $\{ a, c \}, \{b, d\} \notin G$, removes $\{ a, b \}, \{c, d\}$ from the graph and adds $\{ a, c \}, \{b, d\}$ to it, so that all the degrees remain the same.

We can switch any graph with $\{ a, b \} \in G$ to a graph without it, and vice versa. By counting the number of possible ways to preform such a switch and to switch back, we may relate the number of graphs with $\{a, b\}$ as an edge to the number of graphs without it, and obtain an estimate on the probability that it appears.

The number of switching operations that remove $\{a, b\}$ is just the number of ways to choose an ordered edge $e_1 = \{ v_1, v_2 \}$ such that $e_1$ is in $E(G)$ and $e_2 = \{ v_1, a \}, e_3 = \{ v_2, b \}$ are not. There are at least $n d - t$ ordered edges, and at most $d + t + 1$ vertices are forbidden from being $v_1$ or $v_2$ --- either they are part of some edge that is not in $\mathcal{A}$, or they are neighbors of $a$ or $b$, or they themselves are $a$ or $b$. Thus, there are at least $n d - t - 2(d + t + 1)^2 = n d - 2d^2 + o(d^2) \geq (1 - 2 c) n d + o(d^2) \geq (1 - 3c) n d$ such choices of an ordered edge.

We can similarly switch in the other direction, starting from a graph without the edge $\{ a, b \}$, by choosing $e_1 \notin E(G)$ such that $e_2, e_3 \in E(G)$. Since $a, b$ each have $d$ neighbors, there are at most $d^2$ ways to do this.

Therefore, we can consider the ratio $\beta$ between the number of graphs with $ab$ to the number of graphs without it. We know that
\begin{equation*}
    \beta = \frac {P_{\mathbf{d}, \mathcal{A}}(ab)} {1 - P_{\mathbf{d}, \mathcal{A}}(ab)} < \frac {d^2} {(1 - 3c) n d} < (1 + 4c) d / n ,
\end{equation*}
for $c$ small enough. Thus $P_{\mathbf{d}, \mathcal{A}}(ab) \leq \beta / (1 + \beta) \leq (1 + 4c) d / n$.
\end{proof}

We are now ready to estimate various probabilities in $G_{\mathbf{d}, \mathcal{A}}$. In particular, \Cref{first_estimate} suggests that the estimates $\tilde{P}_{\mathbf{d}, \mathcal{A}} (ab) = d / n$ and $\tilde{Y}_{\mathbf{d}, \mathcal{A}} (abc) = d^2 / n^2$ may be reasonable first guesses for $P_{\mathbf{d}, \mathcal{A}}$ and $Y_{\mathbf{d}, \mathcal{A}}$ respectively, for any $\mathbf{d} \in B_t (n, d)$ and $|\mathcal{A}| \geq \binom {n} {2} - t$. This initial estimate may then be gradually refined by iterating \Cref{recursive_estimates}, repeatedly plugging in the previous estimate to obtain an improved one.

Suppose we have estimates
\begin{align*}
    \tilde{P} & \colon \mathbb{Z}^n_{\geq 0} \times 2 ^ {\binom{[n]} {2}} \times [n]^2 \rightarrow \mathbb{R} , \\
    \tilde{Y} & \colon \mathbb{Z}^n_{\geq 0} \times 2 ^ {\binom{[n]} {2}} \times [n]^3 \rightarrow \mathbb{R} .
\end{align*}
We think of $\tilde{P}$ as receiving a degree sequence $\mathbf{d}$, some set $\mathcal{A}$, and two vertices $a, b$, written as $\tilde{P}_{\mathbf{d}, \mathcal{A}} (ab)$, and returning some approximation of $\mathds{P} [\{a, b\} \in G_{\mathbf{d}, \mathcal{A}} ]$. The function $\tilde{Y}$ is similar, but $\tilde{Y}_{\mathbf{d}, \mathcal{A}} (abc)$ takes in three vertices instead of two, and outputs an approximation for $\mathds{P} [\{a, b\}, \{b, c\} \in G_{\mathbf{d}, \mathcal{A}} ]$.

We then define the following operators:
\begin{align}
\mathcal{B}(\tilde{P}, \tilde{Y})_{\mathbf{d}, \mathcal{A}} (ab) & = \sum\limits_{c \in \mathcal{A}(a) \setminus \mathcal{A}(b)} \tilde{P}_{\mathbf{d}, \mathcal{A}} (ac) + \sum\limits_{c \in \mathcal{A}(a) \cap \mathcal{A}(b)} \tilde{Y}_{\mathbf{d}, \mathcal{A}} (acb), \label{b_operator} \\
\mathcal{P}(\tilde{P}, \tilde{Y})_{\mathbf{d}, \mathcal{A}} (ab) & = \mathbf{d} (b) \left( \sum\limits_{c \in \mathcal{A}^*(b)} \frac {\mathbf{d}(c) - \mathcal{B}(\tilde{P}, \tilde{Y})_{\mathbf{d} - \mathbf{e}_a - \mathbf{e}_b, \mathcal{A}} (ca)} {\mathbf{d}(a) - \mathcal{B}(\tilde{P}, \tilde{Y})_{\mathbf{d} - \mathbf{e}_b - \mathbf{e}_c, \mathcal{A}} (ac)} \cdot \frac {1 - \tilde{P}_{\mathbf{d} - \mathbf{e}_b - \mathbf{e}_c, \mathcal{A}}(bc)} {1 - \tilde{P}_{\mathbf{d} - \mathbf{e}_a - \mathbf{e}_b, \mathcal{A}} (ab)} \right)^{-1} , \label{p_operator} \\
\mathcal{Y}(\tilde{P}, \tilde{Y})_{\mathbf{d}, \mathcal{A}} (abc) & = \mathcal{P}(\tilde{P}, \tilde{Y})_{\mathbf{d}, \mathcal{A}} (ab) \frac {\mathcal{P}(\tilde{P}, \tilde{Y})_{\mathbf{d} - \mathbf{e}_a - \mathbf{e}_b, \mathcal{A}} (bc)-Y_{\mathbf{d} - \mathbf{e}_a - \mathbf{e}_b, \mathcal{A}} (abc)} {1 - \mathcal{P}(\tilde{P}, \tilde{Y})_{\mathbf{d} - \mathbf{e}_a - \mathbf{e}_b, \mathcal{A}} (ab)} \label{y_operator} .
\end{align}

Let $\tilde{P}, \tilde{Y}$ be a pair of estimates on all graphs with some $\mathcal{A}$ and any $\mathbf{d} \in B_t(n, d)$. Then, under certain conditions, replacing them with $\mathcal{P}(\tilde{P}, \tilde{Y}), \mathcal{Y}(\tilde{P}, \tilde{Y})$ will give us new estimates on the same $\mathcal{A}$ and any $\mathbf{d} \in B_{t-2}(n, d)$. Applying $\mathcal{P}, \mathcal{Y}$ to those new estimates will then give estimates for $\mathbf{d} \in B_{t-4} (n, d)$, and so on. The repeated application of $\mathcal{P}, \mathcal{Y}$ will yield better and better approximations of $P$ and $Y$, under certain conditions. In fact, it suffices to prove:

\begin{proposition}
\label{good_estimate}
There exists some constant $C$ so that the following holds. Let $n, d, t, \mathcal{A}$ be as in \Cref{first_estimate}. Also, let $(\tilde{P}_1, \tilde{Y}_1), (\tilde{P}_2, \tilde{Y}_2)$ be two estimates in $B_t(n, d)$ such that
\begin{equation*}
    \tilde{P}_i = (1 + o(1)) p, \tilde{Y}_i = (1 + o(1)) p^2 ,
\end{equation*}
for $i \in \{1,2\}$. If $\tilde{P}_1 = (1 \pm \xi) \tilde{P}_2, \tilde{Y}_1 = (1 \pm \xi) \tilde{Y}_2$, for some $0 \leq \xi \leq 1$, then
\begin{equation*}
    \mathcal{P}(\tilde{P}_1, \tilde{Y}_1) = (1 \pm C p \xi) \mathcal{P}(\tilde{P}_2, \tilde{Y}_2), \mathcal{Y}(\tilde{P}_1, \tilde{Y}_1) = (1 \pm C p \xi) \mathcal{Y}(\tilde{P}_2, \tilde{Y}_2) .
\end{equation*}
\end{proposition}

\begin{proof}
This is proven as Lemma 5.1 in \cite{liebenauwormald2017}. However, small modifications are required, so we state the proof here for completeness.

The assumptions on $\tilde{P}_1, \tilde{Y}_1, \tilde{P}_2, \tilde{Y}_2$ imply that, when applying \Cref{b_operator} to the estimates, we will have $\mathcal{B}(\tilde{P}_1, \tilde{Y}_1) = (1 \pm \xi) \mathcal{B}(\tilde{P}_2, \tilde{Y}_2)$. Additionally, for all $\mathbf{d} \in B_t (n, d)$, we have $\mathcal{B}(\tilde{P}_2, \tilde{Y}_2) \leq C_1 d p$. Therefore $\mathcal{B}(\tilde{P}_1, \tilde{Y}_1) = \mathcal{B}(\tilde{P}_2, \tilde{Y}_2) + C_1 d p \xi$.

Applying \Cref{p_operator}, we see that each summand of $\mathcal{P}(\tilde{P}_1, \tilde{Y}_1)$ is $(1 + o(1))$, and differs from the corresponding summand of $\mathcal{P}(\tilde{P}_2, \tilde{Y}_2)$ by a factor of $(1 \pm C_2 p \xi)$, as long as $p < c$ for some $c$ small enough that the denominators are bounded away from zero.

Therefore, the respective sums only differ by a $(1 \pm C_2 p \xi)$ factor, and its inverse differs by at most $(1 \pm C_3 p \xi)$ for $c$ small enough. Since $\mathbf{d}(b)$ remains the same for the two estimates, we see that $\mathcal{P}(\tilde{P}_1, \tilde{Y}_1) = (1 \pm C_3 p \xi) \mathcal{P}(\tilde{P}_2, \tilde{Y}_2)$.

Plugging this into \Cref{y_operator}, and again using the fact that the denominator is bounded away from zero, we see that $\mathcal{Y}(\tilde{P}_1, \tilde{Y}_1) = (1 \pm C_4 p \xi) \mathcal{Y}(\tilde{P}_2, \tilde{Y}_2)$ as well.
\end{proof}

\Cref{good_estimate} essentially suggests that any two approximations grow closer upon repeated application of $\mathcal{P}, \mathcal{Y}$. Note that the actual probabilities $\tilde{P} = P, \tilde{Y} = Y$ must be a fixed point of those iterations, and therefore repeated application of $\mathcal{P}, \mathcal{Y}$ will converge to $P, Y$ from any ``reasonable'' starting point.

This also suggests that $P, Y$ can be estimated by approximating the fixed point of $\mathcal{P}$ and $\mathcal{Y}$. If we show that plugging in ``reasonable'' $\tilde{P}, \tilde{Y}$ within a certain range returns new estimates still within that range, then \Cref{good_estimate} implies that the fixed point $P, Y$ must be within that range as well.

Such approximation is explicitly carried out in Appendix A of \cite{liebenauwormald2017} for $\mathcal{A} = \binom {[n]} {2}$ and $\mathbf{d} \in B_t (n, d)$ for $t \ll \sqrt{d}$. Here we need a slightly more general claim:

\begin{claim}
\label{p_estimate}
For any $\binom {n} {2} - t \leq |\mathcal{A}| \leq \binom {n} {2}$ and $\mathbf{d} \in B_t (n, d)$ for $t \ll \sqrt{d}$, the estimates
\begin{equation*}
    \tilde{P}_{\mathbf{d}, \mathcal{A}} (ab) = \frac {(n - 1) \mathbf{d}(a) \mathbf{d}(b)} {d |\mathcal{A}(a)| |\mathcal{A}(b)|} ,
\end{equation*}
\begin{equation*}
    \tilde{Y}_{\mathbf{d}, \mathcal{A}} (abc) = \frac {(n - 1)^2 \mathbf{d}(a) \mathbf{d}(b) (\mathbf{d}(b) - 1) \mathbf{d}(c)} {d^2 |\mathcal{A}(a)| |\mathcal{A}(b)| (|\mathcal{A}(b)| - 1) |\mathcal{A}(c)|} ,
\end{equation*}
approximate $P, Y$ up to a multiplicative factor of
\begin{equation*}
    \left(1 + \mathcal{O} \left( \frac {t^2} {n d} \right) \right) ,
\end{equation*}
for any $\{a, b\}, \{b, c\} \in \mathcal{A}$.
\end{claim}

While the proof is analogous to the $\mathcal{A} = \binom {[n]} {2}$ case, it is presented in \Cref{first_order_estimate} for completeness.

The following Corollary appears in many papers (e.g.\ \cite{mkcay1985, ksv2007, mckay2011, gaoohapkin2020}) but only for some subset of the possible $d$ values that are relevant to us, or only for specific cases, such as $\mathcal{B} = \emptyset$ or $\mathcal{C} = \emptyset$. We thus prove it here in the general setting we need:

\begin{corollary}
\label{joint_probability_estimate}
Given disjoint $\mathcal{B},\mathcal{C} \subseteq \binom {[n]} {2}$ with $|\mathcal{B} \cup \mathcal{C}| \leq k$, and a random $d$-regular graph $G$, such that $d \leq c n$ and $0 < k \ll \sqrt{d}$, then
\begin{equation*}
    \mathds{P} [ \mathcal{B} \subseteq E(G), \mathcal{C} \cap E(G) = \emptyset] = \left( \frac {d} {n} \right)^{|\mathcal{B}|} \left( 1 - \frac {d} {n} \right)^{|\mathcal{C}|} \left(1 + \mathcal{O} \left( \frac {k^3} {n d} \right) \right) .
\end{equation*}
\end{corollary}

\begin{proof}
As noted in \Cref{conditioning_smaller_graph}, conditioning a random graph with possible edges $\mathcal{A}$ and given degree sequence $\mathbf{d}$ on the existence or non-existence of $ab \in \mathcal{A}$ is equivalent to removing it from $\mathcal{A}$, and also replacing $\mathbf{d}$ with $\mathbf{d} - \mathbf{e}_a - \mathbf{e}_b$ if we condition on its existence.

Therefore, we may start from the full graph $G$ without any conditioning, and repeatedly apply \Cref{p_estimate} to obtain the probability that some edge $e_1 = \{v_1, u_1 \} \in \mathcal{B}$ appears. Conditioning on that appearance, the conditional probability that $e_2 \in \mathcal{B}$ also appears is equivalent to the probability for it to appear in another graph, with $\mathcal{A} = \binom{[n]} {2} \setminus \{e_1 \}$ and $\mathbf{d}$ replaced by $\mathbf{d} - \mathbf{e}_{v_1} - \mathbf{e}_{v_2}$. Therefore, \Cref{p_estimate} applies once again. We can repeat that process until all of $\mathcal{B} \cup \mathcal{C}$ has been conditioned on.

At each step, the estimate we use may have a relative error of at most
\begin{equation*}
    \mathcal{O} \left( \frac {k^2} {n d} \right) .
\end{equation*}
When conditioning on an edge in $\mathcal{C}$, we multiply by $1-P$ instead of $P$, but since $d / n < c$, we may choose small enough $c$ and make sure that the relative error is even smaller in that case.

The resulting conditional probability estimate is therefore $(d + \mathcal{O}(k)) / n$ for each $e \in \mathcal{B}$ and $(n - d + \mathcal{O}(k)) / n$ for each $e \in \mathcal{C}$. Multiplying all of them yields 
\begin{equation*}
\left( \frac {d} {n} \right)^{|\mathcal{B}|} \left( 1 - \frac {d} {n} \right)^{|\mathcal{C}|} \left(1 + \mathcal{O} \left( \frac {k} {d} \right) \right)^{k} = \left( \frac {d} {n} \right)^{|\mathcal{A}|} \left( 1 - \frac {d} {n} \right)^{|\mathcal{C}|} \left(1 + \mathcal{O} \left( \frac {k^3} {n d} \right) \right) ,
\end{equation*}
since $k \ll \sqrt{d} \leq \sqrt{n}$.
\end{proof}

\section{Contribution of Walks}
\label{section_contribution}

As explained in \Cref{our_results_section},
\begin{equation*}
    \Tr \left[ (A - p J + p I)^k \right] = \sum\limits_{\Gamma} \prod\limits_{e \in \Gamma} (\mathds{1}_{e \in G}- p) ,
\end{equation*}
where the sum is over all closed walks $\Gamma$ of length $k$ on $[n]$. Here the product is over $\Gamma$ as a multiset --- if $e_i$ appears in $\Gamma$ with multiplicity $m_i$, it also appears $m_i$ times in the product. The expectation of this trace is the sum of expectations of such terms.

In this section, we will calculate a bound for the order of magnitude of the expectation
\begin{equation*}
    M_{\Gamma} = \mathds{E} \left[ \prod\limits_{e \in \Gamma} (\mathds{1}_{e \in G} - p) \right],
\end{equation*}
for each such walk $\Gamma$.

Throughout this section, $\Gamma$ is a walk of length $k = k(n)$. We make a distinction between \textit{non-repeating edges} that appear in $\Gamma$ only once, and \textit{repeating edges} that appear at least twice.

\subsection{Walks with All Edges Repeating}

We begin by bounding $|M_{\Gamma}|$ for walks with $t$ distinct edges, where every edge $e_i$ appears $m_i > 1$ times. Each term $(\mathds{1}_{e_i \in G}- p)^m$ is $(1-p)^m$ if $e \in G$ and $(-p)^m$ otherwise. Therefore, if $e_1, \ldots, e_t$ are the edges of $\Gamma$, then
\begin{equation}
\label{naive_t_bound}
    |M_{\Gamma}| \leq \sum\limits_{S \subseteq [t]} (1 - p)^{m_S} p^{k - m_S} \mathds{P} [ \bigcap\limits_{i \in S} \{ e_i \in G \} \cap \bigcap\limits_{i \in [t] \setminus S} \{ e_i \notin G \} ],
\end{equation}
where $m_S = \sum_{i \in S} m_i$.

\begin{claim}
\label{all_repeating_t_bound}
If $t^2 \ll d$ and $\Gamma$ is a walk with $t$ distinct edges, where every edge is repeated at least twice, then $|M_{\Gamma}| \leq (1 + o(1)) [p (1-p)] ^ {t}$.
\end{claim}

\begin{proof}
We begin with the bound given in \Cref{naive_t_bound}, and estimate the various probabilities used there. By \Cref{joint_probability_estimate}, for any given $S \subseteq [t]$, we have
\begin{equation*}
    \mathds{P} [ \bigcap\limits_{i \in S} \{ e_i \in G \} \cap \bigcap\limits_{i \in [t] \setminus S} \{ e_i \notin G \} ] = \left(1 + \mathcal{O} \left(\frac {t^3} {n d} \right) \right) p^{|S|} (1 - p)^{t - |S|} = (1 + o(1)) p^{|S|} (1 - p)^{t - |S|} ,
\end{equation*}
since $t^2 \ll d \leq n$.

This shows that
\begin{equation*}
    |M_{\Gamma}| \leq (1 + o(1)) \prod\limits_{i=1}^{t} \left[p (1-p)^{m_i} + (1-p) p^{m_i} \right] .
\end{equation*}
Since $m_i \geq 2$, each term in the product above is at most $p (1-p)^2 + (1-p) p^2 = p(1-p)$, and therefore the contribution of any such walk with $t$ distinct edges is at most $(1 + o(1)) \left[ p (1-p) \right]^t$.
\end{proof}

\subsection{Walks with Some Edges Repeating}
\label{mixed_walks}

When considering walks $\Gamma$ where some edges appear only once, the task of estimating $M_{\Gamma}$ is significantly more complex, since a lot of accurate cancellation is needed when expanding the product of $\mathds{1}_e - p$ terms. This requires a finer estimate than what can be provided by \Cref{first_estimate} or even by \Cref{p_estimate}.

The reasoning behind this is that, when expanding the product of the various $\mathds{1}_e - p$ terms, we see exponentially many factors. Most of those factors should cancel each other out when taking the expected value. If the various edges were independent, the expected product would trivially be $0$, so we may hope that when the edges are ``nearly independent'' the product will still be small.

Consider, then, a closed walk $\Gamma$ with $t$ distinct edges. Denote them by $e_1, \ldots, e_t$ and their multiplicities by $m_1, \ldots, m_t$ like before. Let $T_1 = \{ i : m_i = 1 \}$ and $T_2 = \{ i : m_i > 1 \}$. We denote by $t_1, t_2$ the sizes of $T_1, T_2$ respectively, and use $k_1, k_2$ as shorthand for $m_{T_1}, m_{T_2}$ respectively. In particular, $k_1 = t_1$.

Now, as before,
\begin{equation*}
    M_{\Gamma} = \sum\limits_{S \subseteq [t]} (1-p)^{m_S} (-p)^{k - m_S} \mathds{P} [ \bigcap\limits_{i \in S} \{ e_i \in G \} \cap \bigcap\limits_{i \in [t] \setminus S} \{ e_i \notin G \} ] .
\end{equation*}
We can split the sum according to the intersections $S_1 = S \cap T_1$ and $S_2 = S \cap T_2$:
\begin{equation*}
    M_{\Gamma} = \sum\limits_{S_2 \subseteq T_2} (1-p)^{m_{S_2}} (-p)^{k_2 - m_{S_2}} \mathds{P} [ \bigcap\limits_{i \in S_2} \{ e_i \in G \} \cap \bigcap\limits_{i \in T_2 \setminus S_2} \{ e_i \notin G \} ] \cdot I_{S_2},
\end{equation*}
with the inner sum
\begin{equation*}
    I_{S_2} = \sum\limits_{S_1 \subseteq T_1} (1-p)^{|S_1|} (-p)^{k_1 - |S_1|} \mathds{P} [ \bigcap\limits_{i \in S_1} \{ e_i \in G \} \cap \bigcap\limits_{i \in T_1 \setminus S_1} \{ e_i \notin G \} \mid \bigcap\limits_{i \in S_2} \{ e_i \in G \} \cap \bigcap\limits_{i \in T_2 \setminus S_2} \{ e_i \notin G \} ],
\end{equation*}
since $m_{S_1} = |S_1|$ for any $S_1 \subseteq T_1$.

Taking the absolute value of each summand of the outer sum, we get:
\begin{equation*}
    |M_{\Gamma}| \leq \sum\limits_{S_2 \subseteq T_2} (1-p)^{m_{S_2}} p^{k_2 - m_{S_2}} \mathds{P} [ \bigcap\limits_{i \in S_2} \{ e_i \in G \} \cap \bigcap\limits_{i \in T_2 \setminus S_2} \{ e_i \notin G \} ] \cdot |I_{S_2}|.
\end{equation*}
Given a bound $|I_{S_2}| \leq B_{\Gamma}$ for all possible inner sums, the total sum would then be bounded from above by
\begin{equation*}
    B_{\Gamma} \cdot \sum\limits_{S_2 \subseteq T_2} (1-p)^{m_{S_2}} p^{k_2 - m_{S_2}} \mathds{P} [ \bigcap\limits_{i \in S_2} \{ e_i \in G \} \cap \bigcap\limits_{i \in T_2 \setminus S_2} \{ e_i \notin G \} ] .
\end{equation*}
Now, the same argument of \Cref{all_repeating_t_bound} may be applied to the sum. Indeed, the proof of the claim can be applied to any arbitrary multiset of edges, as long as each edge repeats at least twice. Thus, such a universal bound $B_{\Gamma}$ on all the inner sums would yield a bound $|M_{\Gamma}| \leq (1 + o(1)) B_{\Gamma} [p (1-p)] ^ {t_2}$.

The heart of the matter is a derivation of such a bound $B_{\Gamma}$ on the inner sum, only involving the non-repeating edges. Note that we condition on $e_i \in G$ for all $i \in S_2$ and $e_i \notin G$ for all $i \in T_2 \setminus S_2$. This is equivalent to slightly modifying $\mathbf{d}$ and $\mathcal{A}$, as explained in \Cref{conditioning_smaller_graph}.

Thus, as long as we allow for such small modifications of $\mathcal{A}$ and $\vec{d}$, it suffices to bound
\begin{equation*}
    \sum\limits_{S_1 \subseteq T_1} (1-p)^{|S_1|} p^{k_1 - |S_1|} \mathds{P} [ \bigcap\limits_{i \in S_1} \{ e_i \in G \} \cap \bigcap\limits_{i \in T_1 \setminus S_1} \{ e_i \notin G \} \mid \bigcap\limits_{i \in S_2} \{ e_i \in G \} \cap \bigcap\limits_{i \in T_2 \setminus S_2} \{ e_i \notin G \} ]
\end{equation*}
for various $S_2 \subseteq T_2$. In summary, the general case reduces to the case of $\Gamma$ without any repeating edges, in a slightly modified probability space. Note, however, that after moving to a subset $T_1$ of the edges, what remains of $\Gamma$ is not necessarily a walk, but an arbitrary collection of unique edges.

\subsection{Walks without Repeating Edges}

We now assume for simplicity that $\Gamma$ is just a sequence of $t$ unique edges $e_1, \ldots, e_t$. We cannot assume that consecutive edges are incident to one another, since they may actually be separated by repeating edges from $T_2$ that we ``forgot'' about. We may also be working with a slightly modified probability space, obtained by conditioning on the state of at most $k$ edges.

For any $1 \leq i \leq t$ we can now define $p_i \colon \{0, 1\}^{i-1} \rightarrow \mathds{R}$ by
\begin{equation*}
    p_i(x_1, \ldots, x_{i-1}) = \mathds{P} [e_i \in G \mid \bigcap\limits_{\substack{1 \leq j < i \\ x_j = 1}} \{ e_j \in G \} \cap \bigcap\limits_{\substack{1 \leq j < i \\ x_j = 0}} \{ e_j \notin G \}] .
\end{equation*}
When $i$ is clear from the context, we will similarly think of other graph quantities as functions on the cube $\{ 0, 1 \} ^ {i-1}$, and by abuse of notation we freely apply $\chi_S$ to such quantities. For example, when talking about $\chi_S (Y_{\mathbf{d}, \mathcal{A}} (abc))$, we think of $Y_{\mathbf{d}, \mathcal{A}} (abc)$ as a binary function receiving $(x_1, \ldots, x_{i-1})$ and returning the quantity $Y(abc)$ on the modified graph obtained by conditioning on $e_j \in G$ if $x_j = 1$, and $e_j \notin G$ otherwise.

We can thus measure the effect that previous edges have on the next edge using $\chi_S(p_i)$ for any $S \subseteq [i-1]$. Thus, we may hope for the contribution of the walk to be small when the Fourier coefficients $\chi_S$ are small. This relationship is explicitly quantified in the following proposition:

\begin{proposition}
\label{chi_expansion}
\begin{equation*}
    \mathds{E} \left[ \prod\limits_{i=1}^t (\mathds{1}_{e_i \in G} - p) \right] = \sum\limits_{\substack {S_1, \ldots, S_t \\ S_i \subseteq [i-1]}} \prod\limits_{i=1}^{t} (1 - p \cdot \mathds{1}_{i \in \bigcup\limits_{j > i} S_j}) \chi_{S_i} ( p_i - p \cdot \mathds{1}_{i \notin \bigcup\limits_{j > i} S_j} ) .
\end{equation*}
\end{proposition}

\begin{proof}
Define random variables $X_i = \mathds{1}_{e_i \in G}$. Then $\mathds{E} [ X_i | X_1, X_2, \ldots, X_{i-1} ]$ is, by definition, $p_i (X_1, X_2, \ldots, X_{i-1})$. Since we are calculating the expectation of a product of $X_i - p$ terms, we will start by conditioning on $X_1, \ldots, X_{t-1}$, replacing $X_t$ by $p_t (X_1, \ldots, X_{t-1})$ and expanding $p_t$ as a polynomial in $X_1, \ldots, X_{t-1}$. We can then repeat this process with $X_{t-1}$, $X_{t-2}$, and so on.

Note, however, that we replaced $X_t$ with some linear combination of the monomials $X_{S_t}$ for various $S_t \subseteq [t - 1]$. Therefore, when we now repeat this process for $X_{t-1}$, we deal with $X_{t-1} - p$ when multiplied by some $X_{S_t}$ --- which may itself contain $X_{t-1}$. To deal with such issues, note that $X_i \cdot (X_i - p) = X_i (1 - p)$, since $X_i$ is supported on $\{0, 1\}$. Thus, we define
\begin{equation*}
    Z_{i, r}(S_{r+1}, S_{r+2}, \ldots, S_t) = \left\{\begin{array}{lr}
        X_i \cdot (1 - p), & \text{if } i \in S_{r+1} \cup S_{r+2} \cup \dots \cup S_t \\
        X_i - p, & \text{otherwise}
        \end{array}\right. .
\end{equation*}

Now, we will prove the following fact by descending induction on $r$, from $r = t$ down to $r = 0$:
\begin{multline}
\label{chi_expansion_induction}
    \mathds{E} \left[ \prod\limits_{i=1}^t (X_i - p) \mid X_1, \ldots, X_r \right] \\
    = \sum\limits_{\substack {S_{r+1}, \ldots, S_t \\ S_k \subseteq [k-1]}} \prod \limits_{i=1}^{r} Z_{i, r} (S_{r+1}, S_{r+2}, \ldots, S_t) \prod\limits_{i=r+1}^{t} (1 - p \cdot \mathds{1}_{i \in \bigcup\limits_{j > i} S_j}) \chi_{S_i} ( p_i - p \cdot \mathds{1}_{i \notin \bigcup\limits_{j > i} S_j} ) .
\end{multline}
Note that this will conclude the proof, as the $r = 0$ case is equivalent to \Cref{chi_expansion}.

The induction base $r = t$ follows trivially from the fact that $Z_{i, r} = X_i - p$ by definition. Now, to move from $r$ to $r - 1$, we take the conditional expectation of the right-hand side of \Cref{chi_expansion_induction} with respect to $X_{r}$. In each product, the only term that has $X_r$ is $Z_{r, r} (S_{r+1}, S_{r+2}, \ldots , S_t)$. Therefore, this is the only term that changes, being replaced by
\begin{equation*}
    \mathds{E} \left[ Z_{r, r} (S_{r+1}, S_{r+2}, \ldots , S_t) \mid X_1, X_2, \dots, X_{r - 1} \right] = \left\{\begin{array}{lr}
        p_r \cdot (1 - p), & \text{if } r \in S_{r+1} \cup \dots \cup S_t \\
        p_r - p, & \text{otherwise}
        \end{array}\right. .
\end{equation*}

Calculating the Fourier series for the term above, we see that it can be written as
\begin{equation*}
    \sum\limits_{S_r \subseteq [r - 1]} X_{S_r} (1 - p \cdot \mathds{1}_{r \in \bigcup\limits_{j > r} S_j}) \chi_{S_r} ( p_r - p \cdot \mathds{1}_{r \notin \bigcup\limits_{j > r} S_j} ) .
\end{equation*}
Note that the $X_{S_r}$ term is exactly what we need in order to replace every $Z_{i, r} (S_{r+1}, S_{r+2}, \ldots, S_t)$ with $Z_{i, r - 1} (S_r, S_{r+1}, \ldots, S_t)$.

Indeed, we can split this $X_{S_r}$ and use it to have any $Z_{i, r} (S_{r+1}, S_{r+2}, \ldots, S_t)$ term with $i \in S_r$ be multiplied by $X_i$. This will change all those terms to be $X_i (1 - p)$ not only when $i \in S_{r + 1} \cup S_{r + 2} \cup \dots \cup S_t$ but also when $i \in S_r$, while the terms with $i \notin S_r \cup \dots cup S_t$ as they are.

When summed over all $S_r \in [r-1]$, this gives us the correct term for $r - 1$, thus proving the induction step.
\end{proof}

Given this, we hope to prove that the events $\{ e_i \in G \}$ are close enough to being independent, in the sense that the Fourier coefficients $\chi_S (p_i)$ are small. This would imply a bound on the expectation above. Indeed, we have:

\begin{proposition}
\label{recursive_chi_estimate}

Let $t = t(n)$ satisfy $t^{7} \ll d < c n$ for small enough $c$, $\mathbf{d} \in B_t(n, d)$, and $\mathcal{A} \subseteq \binom {n} {2}$ missing at most $t$ edges. Fix $1 \leq i \leq t$, let $e_1, \ldots, e_i \in \mathcal{A}$, and let $S \subseteq [i-1]$. Then,
\begin{equation*}
    |\chi_S (p_i)| \leq \frac {d} {n} \frac {(4|S|)!} {(24 d)^{|S|}} \left(1 + \mathcal{O} \left( \frac {t^2} {d} \right) \right) .
\end{equation*}
\end{proposition}

\begin{proof}
We begin by plugging in some initial estimates for $P, Y$ across various graphs. By recursively applying the estimation iterations of \Cref{good_estimate}, we will obtain gradually improving estimates for $P, Y$, and we will show by induction that our estimates of $P$ satisfy the required condition, and that our estimates of $Y$ satisfy a similar condition as well. Since those estimates produce gradually improving approximations of the real $P, Y$, this will eventually mean that the real $P, Y$ satisfy these conditions as well.

One issue needs to be addressed before we can start: \Cref{good_estimate} turns an initial approximation for $P, Y$ on $B_t (n, d)$ into a better approximation on $B_{t - 2} (n, d)$. Thus, if we want to use this iterative approximation $i$ times and to obtain an estimate for $P, Y$ on $B_t (n, d)$, we will need to start with an initial estimate on $B_{t + 2i} (n, d)$ and to work recursively. In particular, the more iterations we have, the more degree sequences we need to consider.

This can be controlled by taking $i = \Theta (|S| \log_{n/d} n)$ iterations: we know from \Cref{good_estimate} that this amount of iterations can give us an approximation with error $(C p)^i = o (n^{-100 |S|})$, and by \Cref{fourier_formula} this translates to a negligible error in $|\chi_S (p_t)|$. On the other hand, this amount of iterations is still $o(d^{1/6})$, so we also have $(t + 2i) ^ 6 \ll d$. This will be crucial for the application of \Cref{prod_bound}, \Cref{inv_bound} and \Cref{good_estimate}.

Note that \Cref{prod_bound} and \Cref{inv_bound} require working with functions satisfying $f(0, 0, \dots, 0) = 1$. In order to apply those propositions to functions with any nonzero value of $f(0, 0, \dots, 0)$, we will simply work with $f(x_1, x_2, \dots, x_t) / f(0, 0, \dots, 0)$ instead of $f(x_1, x_2, \dots, x_t)$. For those normalized functions we will then prove a bound of some fixed term times $(24 d)^{-|S|} (4 S)!$ on the Fourier coefficient $\chi_S$, which means we can freely apply \Cref{prod_bound} and \Cref{inv_bound} since $(t + 2i)^6 \ll d$.

For a product $f \cdot g$ we will normalize both $f$ and $g$ this way, bound the Fourier coefficients as described above, apply \Cref{prod_bound}, and then multiply the obtained bounds by $f(0, 0, \dots, 0) \cdot g(0, 0, \dots, 0)$. For an inverse, \Cref{inv_bound} will give a bound on the Fourier coefficients of $f(0, 0, \dots, 0) / f(x_1, x_2, \dots x_t)$, so we will then need to divide by $f(0, 0, \dots, 0)$ to obtain a bound on the Fourier coefficients of $1 / f(x_1, x_2, \dots, x_n)$ itself. Note that we only consider Fourier coefficients of nonzero functions here.

Now, begin with the estimates $\tilde{P}, \tilde{Y}$ given by \Cref{p_estimate} on all $\mathbf{d} \in B_{t + 2i} (n, d)$. For a sufficiently small choice of $c$, it follows from \Cref{first_estimate} that those estimates are at most twice as large as the actual $P, Y$ for those graphs. Thus, we can apply \Cref{good_estimate} to repeatedly replace $\tilde{P}, \tilde{Y}$ with $\tilde{P'} = \mathcal{P} (\tilde{P}, \tilde{Y}), \tilde{Y'} = \mathcal{Y} (\tilde{P}, \tilde{Y})$ and to obtain better and better estimates.

We would like to show that some bounds on $\chi_S$ from the previous iteration carry over to the next one. We will assume we have, for some estimates $\tilde{P}, \tilde{Y}$ and for any $S \subseteq [t]$,
\begin{equation}
\label{iterative_p_tilde_bound}
    |\chi_S ( \tilde{P}_{\mathbf{d}, \mathcal{A}}(ab) ) | \leq \frac {d} {n} (24 d) ^ {-|S|} (4 |S|)! \left(1 + \mathcal{O} \left( \frac {t^2} {d} \right) \right) ,
\end{equation}
\begin{equation}
\label{iterative_y_tilde_bound}
    |\chi_S ( \tilde{Y}_{\mathbf{d}, \mathcal{A}}(abc) ) | \leq C \frac {d^2} {n^2} (24 d) ^ {-|S|} (4 |S|)! ,
\end{equation}
for all $\{a,b\}, \{b,c\} \in \mathcal{A} \setminus \{ e_1, e_2, \dots, e_t \}$, and we wish to show that the same will hold, for any $S \subseteq [t]$, with the improved estimates $\tilde{P'} = \mathcal{P} (\tilde{P}, \tilde{Y}), \tilde{Y'} = \mathcal{Y} (\tilde{P}, \tilde{Y})$.

For $|S| < 2$, the fact that we start from the estimates of \Cref{p_estimate}, and can only improve upon them, means that $\tilde{P}, \tilde{Y}$ may only differ from those initial estimates by a factor of
\begin{equation*}
    1 + \mathcal{O} \left( \frac {t^2} {nd} \right) .
\end{equation*}
This remains true throughout all iterations, which proves the claim for $|S| < 2$.

Now note that in order to deal with the $|S| \geq 2$ case, we still need to keep track of the Fourier coefficients of such $S$ when applying the definitions of $\mathcal{P}, \mathcal{Y}$ in order to apply \Cref{prod_bound} and \Cref{inv_bound}.

Our plan of attack is as follows: we first verify that the initial approximations given by \Cref{p_estimate} satisfy the bounds of \Cref{iterative_p_tilde_bound} and \Cref{iterative_y_tilde_bound} for all $S$. It can be verified that any non-zero Fourier coefficients for the initial approximation $\tilde{P}$ must have $|S| \leq 2$, and $|S| \leq 4$ for any non-zero Fourier coefficient of $\tilde{Y}$. Explicit calculation of those coefficients then verifies that the bounds indeed hold. It remains to show that the bounds are maintained when replacing $\tilde{P}, \tilde{Y}$ with $\mathcal{P} (\tilde{P}, \tilde{Y}), \mathcal{Y} (\tilde{P}, \tilde{Y})$. Once this is proven, after $i = \Theta (|S| \log_{n/d} n)$ iterations we will obtain estimates accurate enough to be plugged into \Cref{fourier_formula}, thus providing accurate bounds on the Fourier coefficients of the actual probability function $P$ as well.

Using the fact that $|\mathcal{A} (a)| = n - o(d)$ for any $a$ and $|S| = o(d)$, we can obtain bounds on the Fourier coefficients $\chi_S (\mathcal{B}(\tilde{P}, \tilde{Y})_{\mathbf{d}, \mathcal{A}} (ab))$. Recall the definition of $\mathcal{B}$ from \Cref{b_operator},
\begin{equation*}
    \mathcal{B}(\tilde{P}, \tilde{Y})_{\mathbf{d}, \mathcal{A}} (ab) = \sum\limits_{c \in \mathcal{A}(a) \setminus \mathcal{A}(b)} \tilde{P}_{\mathbf{d}, \mathcal{A}} (ac) + \sum\limits_{c \in \mathcal{A}(a) \cap \mathcal{A}(b)} \tilde{Y}_{\mathbf{d}, \mathcal{A}} (acb) ,
\end{equation*}
and plugging in the bounds of \Cref{iterative_p_tilde_bound} and \Cref{iterative_y_tilde_bound} gives
\begin{align*}
    | \chi_{S} (\mathcal{B}(\tilde{P}, \tilde{Y})_{\mathbf{d}, \mathcal{A}} (ab)) | & \leq \sum\limits_{c \in \mathcal{A}(a) \setminus \mathcal{A}(b)} | \chi_{S} \tilde{P}_{\mathbf{d}, \mathcal{A}} (ac) | + \sum\limits_{c \in \mathcal{A}(a) \cap \mathcal{A}(b)} | \chi_{S} \tilde{Y}_{\mathbf{d}, \mathcal{A}} (acb) | \\
    & \leq o(d) \cdot \frac {d} {n} (24 d) ^ {-|S|} (4 |S|)! + (n - o(d)) C_1 \frac {d^2} {n^2} (24 d) ^ {-|S|} (4 |S|)! \\
    & \leq C_2 \frac {d^2} {n} (24 d) ^ {-|S|} (4 |S|)! .
\end{align*}
In particular, we see that for any $\mathbf{d} \in B_{t + 2i - 1} (n, d)$, $\tilde{B} = \mathcal{B} (\tilde{P}, \tilde{Y})$ satisfies $\tilde{B}_{\mathbf{d}, \mathcal{A}} (ab) \leq C_2 {d^2} / {n}$.

We now recall that for any $\mathbf{d} \in B_{t + 2i - 2} (n, d)$, the operator $\mathcal{P}$ was defined in \Cref{p_operator} as
\begin{align*}
\mathcal{P}(\tilde{P}, \tilde{Y})_{\mathbf{d}, \mathcal{A}} (ab) & = \mathbf{d} (b) \left( \sum\limits_{c \in \mathcal{A}^*(b)} \frac {\mathbf{d}(c) - \mathcal{B}(\tilde{P}, \tilde{Y})_{\mathbf{d} - \mathbf{e}_a - \mathbf{e}_b, \mathcal{A}} (ca)} {\mathbf{d}(a) - \mathcal{B}(\tilde{P}, \tilde{Y})_{\mathbf{d} - \mathbf{e}_b - \mathbf{e}_c, \mathcal{A}} (ac)} \cdot \frac {1 - \tilde{P}_{\mathbf{d} - \mathbf{e}_b - \mathbf{e}_c, \mathcal{A}}(bc)} {1 - \tilde{P}_{\mathbf{d} - \mathbf{e}_a - \mathbf{e}_b, \mathcal{A}} (ab)} \right)^{-1} \\
= \mathbf{d} (a) \mathbf{d} (b) & \left( \sum\limits_{c \in \mathcal{A}^*(b)} \mathbf{d} (a) \frac {\mathbf{d}(c) - \mathcal{B}(\tilde{P}, \tilde{Y})_{\mathbf{d} - \mathbf{e}_a - \mathbf{e}_b, \mathcal{A}} (ca)} {\mathbf{d}(a) - \mathcal{B}(\tilde{P}, \tilde{Y})_{\mathbf{d} - \mathbf{e}_b - \mathbf{e}_c, \mathcal{A}} (ac)} \cdot \frac {1 - \tilde{P}_{\mathbf{d} - \mathbf{e}_b - \mathbf{e}_c, \mathcal{A}}(bc)} {1 - \tilde{P}_{\mathbf{d} - \mathbf{e}_a - \mathbf{e}_b, \mathcal{A}} (ab)} \right)^{-1} ,
\end{align*}
and consider the summand for any $c \in \mathcal{A}^{*}(b)$. It can be written as the product of three factors,
\begin{equation*}
    \frac {1 - \tilde{P}_{\mathbf{d} - \mathbf{e}_b - \mathbf{e}_c, \mathcal{A}}(bc)} {1 - \tilde{P}_{\mathbf{d} - \mathbf{e}_a - \mathbf{e}_b, \mathcal{A}} (ab)} \cdot \left( 1 - \frac {\mathcal{B}(\tilde{P}, \tilde{Y})_{\mathbf{d} - \mathbf{e}_b - \mathbf{e}_c, \mathcal{A}} (ac)} {\mathbf{d}(a)} \right)^{-1} \cdot (\mathbf{d}(c) - \mathcal{B}(\tilde{P}, \tilde{Y})_{\mathbf{d} - \mathbf{e}_a - \mathbf{e}_b, \mathcal{A}} (ca)) ,
\end{equation*}
and we would like to bound the Fourier coefficients for each of them.

Starting with the first one, we have $\chi_S (1) = 0$ for $S \neq \emptyset$, and therefore
\begin{equation*}
    |\chi_S ( 1 - \tilde{P}_{\mathbf{d} - \mathbf{e}_a - \mathbf{e}_b, \mathcal{A}} (ab) )| = |\chi_S ( \tilde{P}_{\mathbf{d} - \mathbf{e}_a - \mathbf{e}_b, \mathcal{A}} (ab) )| \leq \frac {d} {n} (24 d) ^ {-|S|} (4 |S|)!  \left(1 + \mathcal{O} \left( \frac {t^2} {d} \right) \right) .
\end{equation*}
An analogous claim holds for the term $1 - \tilde{P}_{\mathbf{d} - \mathbf{e}_a - \mathbf{e}_b, \mathcal{A}} (ab)$ as well. The Fourier coefficient of $S = \emptyset$ is $(1 + o(1)) (1 - p)$ for both of those terms. Applying \Cref{inv_bound} and \Cref{prod_bound} with $b = d / (n - d)$, we get a bound on the Fourier coefficients of the entire term. The resulting bounds for $S \neq \emptyset$ will have this $b = d / (n - d)$ factor, times the $S = \emptyset$ term which is $1 + o(1)$, times $(24 d)^{-|S|} (4 |S|)!$. For $S \neq \emptyset$ we therefore have
\begin{equation}
\label{iteration_summand_term1}
    \left| \chi_S \left( \frac {1 - \tilde{P}_{\mathbf{d} - \mathbf{e}_b - \mathbf{e}_c, \mathcal{A}}(bc)} {1 - \tilde{P}_{\mathbf{d} - \mathbf{e}_a - \mathbf{e}_b, \mathcal{A}} (ab)} \right) \right| \leq C_3 \frac {d} {n - d} (24 d)^{-|S|} (4 |S|)! .
\end{equation}

Similarly, whenever we condition on the existence of an edge and move to a graph with coordinate-wise smaller $\mathbf{d}$, the degree $\mathbf{d}(v)$ of any incident vertex will decrease by $1$. With the usual abuse of notation, we write this as $\chi_S (\mathbf{d} (v)) = -1$ if $S = \{j\}$ for some $e_j$ incident to $v$, and for any other $S \neq \emptyset$ we have $\chi_S (\mathbf{d} (v)) = 0$.

This implies that a similar line of reasoning can be used for the second factor, giving a bound of $(1 + o(1)) n / (n - d)$ on the coefficient with $S = \emptyset$, and for any $S \neq \emptyset$ we again multiply by a $b = d / (n - d)$ factor to get
\begin{equation}
\label{iteration_summand_term2}
    \left| \chi_S \left( 1 - \frac {\mathcal{B}(\tilde{P}, \tilde{Y})_{\mathbf{d} - \mathbf{e}_b - \mathbf{e}_c, \mathcal{A}} (ac)} {\mathbf{d}(a)} \right)^{-1} \right| \leq C_4 \frac {n d} {(n - d)^2} (24 d)^{-|S|} (4 |S|)! .
\end{equation}

For the third factor, we first assume that $c$ is not incident to $e_j$ for any $j \in S$. Note that the vast majority of choices for $c$ satisfy this assumption, with at most $\mathcal{O}(t)$ exceptions. In that case, for $S = \emptyset$ we get a bound of $(1 + o(1)) d (n - d) / n$, and for any $S \neq \emptyset$ we have
\begin{equation}
\label{iteration_summand_term3}
    |\chi_S ( \mathbf{d}(c) - \tilde{B}_{\mathbf{d} - \mathbf{e}_a - \mathbf{e}_b, \mathcal{A}} (ab) ) | = |\chi_S ( \tilde{B}_{\mathbf{d} - \mathbf{e}_a - \mathbf{e}_b, \mathcal{A}} (ab) ) | \leq C_5 \frac {d^2} {n} (24 d) ^ {-|S|} (4 |S|)! .
\end{equation}

Applying \Cref{prod_bound} to the bounds of \Cref{iteration_summand_term1}, \Cref{iteration_summand_term2}, and \Cref{iteration_summand_term3}, we get a bound on the summand for all but $\mathcal{O}(t)$ choices of $c$. The coefficient of $S = \emptyset$ is $(1 + o(1)) d$, and for any $S \neq \emptyset$ we have
\begin{equation}
\label{iteration_summand_bound}
    \left| \chi_S \left( \mathbf{d}(a) \frac {\mathbf{d}(c) - \tilde{B}_{\mathbf{d} - \mathbf{e}_a - \mathbf{e}_b, \mathcal{A}} (ca)} {\mathbf{d}(a) - \tilde{B}_{\mathbf{d} - \mathbf{e}_b - \mathbf{e}_c, \mathcal{A}} (ac)} \cdot \frac {1 - \tilde{P}_{\mathbf{d} - \mathbf{e}_b - \mathbf{e}_c, \mathcal{A}}(bc)} {1 - \tilde{P}_{\mathbf{d} - \mathbf{e}_a - \mathbf{e}_b, \mathcal{A}} (ab)} \right) \right| \leq C_6 \frac {d} {n - d} d (24 d) ^ {-|S|} (4 |S|)! .
\end{equation}
Note that we may bound $C_6 d / (n - d)$ by an arbitrarily small constant $C_7$, by choosing $d < c n$ with $c > 0$ small enough.

If our assumption on $c$ does not hold, we might have $\chi_S (\mathbf{d}(c)) = -1$, which would multiply the bound of \Cref{iteration_summand_term3} for the third term by a factor of at most $(n - d) / d$, thus multiplying the overall bound of \Cref{iteration_summand_bound} by this factor as well. Since there are only $\mathcal{O}(t)$ such ``bad'' choices of $c$, this will not significantly change the sum over all $c$. The number of possible $c$ is $|\mathcal{A}^{*}(b)| = (1 + o(1)) n$, so summing this over all $c \in \mathcal{A}$ we get a sum which is $(1 + o(1)) n d$, with Fourier coefficients
\begin{equation*}
    \left| \chi_S \left( \sum\limits_{c \in \mathcal{A}^{*}(b)} \mathbf{d}(a) \frac {\mathbf{d}(c) - \tilde{B}_{\mathbf{d} - \mathbf{e}_a - \mathbf{e}_b, \mathcal{A}} (ca)} {\mathbf{d}(a) - \tilde{B}_{\mathbf{d} - \mathbf{e}_b - \mathbf{e}_c, \mathcal{A}} (ac)} \cdot \frac {1 - \tilde{P}_{\mathbf{d} - \mathbf{e}_b - \mathbf{e}_c, \mathcal{A}}(bc)} {1 - \tilde{P}_{\mathbf{d} - \mathbf{e}_a - \mathbf{e}_b, \mathcal{A}} (ab)} \right) \right| \leq C_7 n d (24 d) ^ {-|S|} (4 |S|)!
\end{equation*}
for $S \neq \emptyset$.

Taking the inverse of this sum, we get a term which is of size $(1 + o(1)) (n d)^{-1}$, but by \Cref{inv_bound} it has
\begin{equation}
\label{iteration_inverse_sum_bound}
    \left| \chi_S \left( \sum\limits_{c \in \mathcal{A}^{*}(b)} \mathbf{d}(a) \frac {\mathbf{d}(c) - \tilde{B}_{\mathbf{d} - \mathbf{e}_a - \mathbf{e}_b, \mathcal{A}} (ca)} {\mathbf{d}(a) - \tilde{B}_{\mathbf{d} - \mathbf{e}_b - \mathbf{e}_c, \mathcal{A}} (ac)} \cdot \frac {1 - \tilde{P}_{\mathbf{d} - \mathbf{e}_b - \mathbf{e}_c, \mathcal{A}}(bc)} {1 - \tilde{P}_{\mathbf{d} - \mathbf{e}_a - \mathbf{e}_b, \mathcal{A}} (ab)} \right)^{-1} \right| \leq C_8 \frac {1} {n d} (24 d) ^ {-|S|} (4 |S|)!
\end{equation}
with $C_8$ arbitrarily small. Note that these are just the Fourier coefficients of $P_{\mathcal{A}, \mathbf{d}} (ab) / (\mathbf{d}(a) \mathbf{d}(b))$.

Finally, we multiply this by $\mathbf{d} (a) \mathbf{d} (b)$. The degrees $\mathbf{d}(a), \mathbf{d}(b)$ behave linearly when conditioning on specific edges existing or not existing, simply decreasing by $1$ when conditioning on an incident edge. Therefore, the $\mathbf{d} (a) \mathbf{d} (b)$ term has $\chi_S = 0$ for any $S$, unless $S$ has at most one edge incident to $a$ and at most one edge incident to $b$, in which case $\chi_S (\mathbf{d}(a) \mathbf{d}(b)) = (1 + o(1)) d^{2 - |S|}$. Note that no edge in $S$ may be incident to both $a$ and $b$, since we assume $\{a, b\} \in \mathcal{A} \setminus \{ e_1, e_2, \ldots, e_t \}$.

Applying the formula for Fourier coefficients of a product, we see that
\begin{equation*}
    | \chi_S (P_{\mathbf{d}, \mathcal{A}} (ab)) | \leq \sum\limits_{S_1 \cup S_2 = S} \left| \chi_{S_1} \left( \mathbf{d}(a) \mathbf{d}(b) \right) \right| \cdot \left| \chi_{S_2} \left(\frac {P_{\mathbf{d}, \mathcal{A}} (ab)} {\mathbf{d}(a) \mathbf{d} (b)} \right) \right| .
\end{equation*}
For previous terms, we also needed to keep track of the coefficients corresponding to $|S| < 2$, in order for \Cref{prod_bound} and \Cref{inv_bound} to be applicable. But this is no longer needed, so we can now use our assumption that $|S| \geq 2$.

We know that the Fourier coefficients of $P_{\mathbf{d}, \mathcal{A}} (ab) / (\mathbf{d}(a) \mathbf{d} (b))$ are precisely what we bound in \Cref{iteration_inverse_sum_bound}. Using this bound, we see that the $S_1 = \emptyset$ term contributes at most $(1 + o(1)) C_8 d^2 (nd) ^{-1} (24 d) ^ {-|S|} (4 |S|)!$ to the sum, each term with $|S_1| = 1$ may contribute at most $(1 + o(1)) C_8 d (nd)^{-1} (24 d) ^ {1-|S|} (4 |S| - 4)!$, and each term with $|S_1| = 2$ may contribute at most $(1 + o(1)) (nd)^{-1} (24 d) ^ {2-|S|} (4 |S| - 8)!$ --- without the $C_8$ factor since we may have $S_2 = \emptyset$ in this case.

Since those are the only possible nonzero terms, straightforward calculation now yields
\begin{align*}
    |\chi_S \tilde{P}'_{\mathbf{d}, \mathcal{A}} (ab)| \leq & \frac {d} {n} (24 d) ^ {-|S|} (4 |S|)! \cdot \left(C_8 + 2 C_8 |S| \binom {4 |S|} {4} ^ {-1} + \frac {4 |S|^2} {70} \binom {4 |S|} {8} ^ {-1} + o(1) \right) \\
    \leq & \frac {d} {n} (24 d) ^ {-|S|} (4 |S|)! \cdot (2 C_8 + 1/2) ,
\end{align*}
since $|S| \geq 2$. Recall that by taking $c$ small enough, we can make $C_8$ small enough to make sure that $2 C_8 + 1/2 < 1$, thus proving the required bound.

The proof for $\mathcal{Y}$ is much simpler --- recall that this operator is defined in \Cref{y_operator} as
\begin{equation*}
    \mathcal{Y}(\tilde{P}, \tilde{Y})_{\mathbf{d}, \mathcal{A}} (abc) = \mathcal{P}(\tilde{P}, \tilde{Y})_{\mathbf{d}, \mathcal{A}} (ab) \frac {\mathcal{P}(\tilde{P}, \tilde{Y})_{\mathbf{d} - \mathbf{e}_a - \mathbf{e}_b, \mathcal{A}} (bc)-Y_{\mathbf{d} - \mathbf{e}_a - \mathbf{e}_b, \mathcal{A}} (abc)} {1 - \mathcal{P}(\tilde{P}, \tilde{Y})_{\mathbf{d} - \mathbf{e}_a - \mathbf{e}_b \mathcal{A}} (ab)} ,
\end{equation*}
for any $\mathbf{d} \in B_{t + 2i - 2}(n, d)$. We may again deal with each of those $3$ factors separately.

Using what we have just proven about $\mathcal{P}$ and the induction hypothesis, we see that
\begin{equation*}
    |\chi_S (\mathcal{P}(\tilde{P}, \tilde{Y})_{\mathbf{d}, \mathcal{A}} (ab))| \leq C_9 \frac {d} {n} (24 d)^{-|S|} (4|S|)! ,
\end{equation*}
\begin{align*}
    |\chi_S (\mathcal{P}(\tilde{P}, \tilde{Y}) & _{\mathbf{d} - \mathbf{e}_a - \mathbf{e}_b, \mathcal{A}} (bc)-Y_{\mathbf{d} - \mathbf{e}_a - \mathbf{e}_b, \mathcal{A}} (abc))| \\
    & \leq |\chi_S (\mathcal{P}(\tilde{P}, \tilde{Y})_{\mathbf{d} - \mathbf{e}_a - \mathbf{e}_b, \mathcal{A}} (bc))| + |\chi_S (Y_{\mathbf{d} - \mathbf{e}_a - \mathbf{e}_b, \mathcal{A}} (abc))| \\
    & \leq C_{10} \left(\frac {d} {n} + \frac {d^2} {n^2} \right) (24 d)^{-|S|} (4|S|)! \leq 2 C_{10} \frac {d} {n} (24 d)^{-|S|} (4|S|)!,
\end{align*}
\begin{equation*}
    |\chi_S (1 - \mathcal{P}(\tilde{P}, \tilde{Y})_{\mathbf{d}, \mathcal{A}} (ab))^{-1}| \leq C_{11} \frac {n} {n - d} (24 d)^{-|S|} (4|S|)! \leq 2 C_{11} (24 d)^{-|S|} (4|S|)!,
\end{equation*}
for all $S \neq \emptyset$. Here the last inequality follows by applying \Cref{inv_bound} with $b = d / (n - d)$ as before.

Applying \Cref{prod_bound} to the three inequalities above, we get a similar bound for the product of those terms. But by \Cref{y_operator}, this product is precisely $\mathcal{Y}(\tilde{P}, \tilde{Y})_{\mathbf{d}, \mathcal{A}} (abc)$. We therefore get:
\begin{equation*}
    |\chi_S (\mathcal{Y}(\tilde{P}, \tilde{Y})_{\mathbf{d}, \mathcal{A}} (abc))| \leq C_{12} \frac {d^2} {n^2} (24 d)^{-|S|} (4|S|)! . \qedhere
\end{equation*}
\end{proof}

Note that we used \Cref{p_estimate} to obtain the estimate for the $|S| = 1$ case. When $|S| = \{ e_j \}$ and $e_i, e_j$ are non-incident, we can do a bit better:

\begin{proposition}
\label{better_p_estimate}
For $t^6 \ll d$ and $e_i, e_j$ non-incident,
\begin{equation*}
    | \chi_{\{ e_j \}} (p_i) | \ll \frac {1} {t^4 n} .
\end{equation*}
\end{proposition}

\begin{proof}
Denote $e_i = \{ a, b \}$. If $e_i$ and $e_j$ were incident, say at $a$, then conditioning on $e_j$ existing or not existing would have changed $\mathbf{d}(a)$ in the remaining graph. Since they are not, \Cref{p_estimate} gives the same estimate on the probability for $e_i$ existing in $G_{\mathbf{d}, \mathcal{A}}$, both when conditioning on $e_j$ existing and when conditioning on it not existing:
\begin{equation*}
    P_{\mathbf{d}, \mathcal{A}} (ab) = \frac {(n - 1) \mathbf{d}(a) \mathbf{d}(b)} {d |\mathcal{A}(a)| |\mathcal{A}(b)|} \left(1 + \mathcal{O}\left( \frac {t^2} {n d} \right) \right) .
\end{equation*}

Thus, any possible difference will come from the error term of the estimate, and therefore the difference is $\mathcal{O}(t^2 / n^2)$. Since $t^6 \ll d \leq n$, we have $| \chi_{\{ e_j \}} (p_i) | = \mathcal{O}(t^2 n^{-2}) \ll t^{-4} n^{-1}$.
\end{proof}

Recall that $\Gamma = \{e_1, e_2, \ldots, e_t\}$ is a sequence of distinct edges, the result of a possible removal of some repeating edges from a closed walk. We say that $e_i$ returns $\Gamma$ to a \textit{previously discovered vertex} when it is incident to some $e_j$ with $j < i - 1$ and $(i, j) \neq (t, 1)$. Using this terminology, we can combine the last three results to yield the following:

\begin{proposition}
\label{non_repeating_bound}
Let $\Gamma$ be a sequence of $t$ non-repeating edges, with $\log n \ll t \ll d^{1 / 10}$ and $d < c n$. Also, let $m$ be the number of edges that return $\Gamma$ to a previously discovered vertex. Then $| M_{\Gamma} | \leq (2 + o(1)) t^{2m} [p (1-p) / n]^{t/2}$.
\end{proposition}

\begin{proof}
Note that $t^{10} \ll d$, so both \Cref{recursive_chi_estimate} and \Cref{better_p_estimate} are applicable here. We begin by replacing each $\chi$ term in the sum given by \Cref{chi_expansion} with a bound from one of those propositions, and try and bound the resulting sum.

To elaborate, \Cref{chi_expansion} implies that one can bound the product we are interested in using bounds on Fourier coefficients:
\begin{equation*}
| M_{\Gamma} | = \left| \mathds{E} \left[ \prod\limits_{i=1}^t (\mathds{1}_{e_i \in G} - p) \right] \right| \leq \sum\limits_{\substack {S_1, \ldots, S_t \\ S_i \subseteq [i-1]}} \prod\limits_{i=1}^{t} | \chi_{S_i} ( p_i - p \cdot \mathds{1}_{i \notin \bigcup\limits_{j > i} S_j} ) | .
\end{equation*}
We then bound each of those Fourier coefficients with a bound from either \Cref{recursive_chi_estimate} or \Cref{better_p_estimate}. For most choices of $S_i$ we will simply plug in the bound from \Cref{recursive_chi_estimate}:
\begin{equation*}
    |\chi_{S_i} (p_i)| \leq \frac {d} {n} \frac {(4|S|)!} {(24 d)^{|S|}} \left(1 + \mathcal{O} \left( \frac {t^2} {d} \right) \right) ,
\end{equation*}
apart from the following two exceptions:
\begin{enumerate}
    \item When $S_i = \emptyset$ and $i \notin \bigcup\limits_{j > i} S_j$, we get an improved bound of $|\chi_{\emptyset} (p_i - p)| = \mathcal{O}(t^2 / d)$.
    \item When $S_i = {j}$ with $e_i, e_j$ non-incident, we can use \Cref{better_p_estimate} instead of \Cref{recursive_chi_estimate}, getting a bound of $t^{-4} / n$ instead of $1 / n$.
\end{enumerate}
For any choice $S_1, S_2, \ldots, S_t$, we denote the product of those bounds over $1 \leq i \leq t$ by $\Psi_{S_1, S_2, \ldots, S_t}$.

The possible choices $S = (S_1, S_2, \ldots, S_t)$ take value in the set $D = 2 ^ {\emptyset} \times 2 ^ {[1]} \times 2 ^ {[2]} \times \cdots \times 2 ^ {[t-1]}$. We would like to eventually claim that a $1 - o(1)$ fraction of the sum comes from a very small subset of $D$. To that end, we will repeatedly find transformations $f \colon D \setminus C \rightarrow D$ for some smaller subset $C \subseteq D$, such that $\Psi_{S_1, S_2, \ldots, S_t} \ll \Psi_{f(S_1, S_2, \ldots, S_t)}$.

For such $f$, denote by $N$ an upper bound on the maximal number of different preimages of any $S$ under $f$, and let $M$ be some value such that $M \Psi_{S} \leq \Psi_{f(S)}$ for all $S \in D \setminus C$. Then, as long as $N / M = o(1)$, the contribution of all possible sources $S$ will be negligible compared to that of all images $f(S)$. Applying this repeatedly, we see that $1 - o(1)$ of the sum must come from the terms for which we cannot apply $f$ any further --- that is, from $C$.

First, we apply this technique with
\begin{equation*}
    C = \{ (S_1, S_2, \ldots, S_t) \in D : S_i \cap S_j = \emptyset \} .
\end{equation*}
Indeed, define $f$ on $D \setminus C$ by taking the smallest $a \in S_i \cap S_j$ and removing it from the first $S_i$ in which it appears. This will not change any of the $\mathds{1}_{i \in \bigcup S_j}$ terms, but will multiply the $i$th bound by a factor of at least $M = \Omega(d / |S_i|^4) = \Omega(d /t^4)$. To determine a possible preimage $f^{-1}$ we have $t$ choices for $a$ and $t$ choices for the $S_i$ we removed it from, so we can take $N = t^2$. Since $N / M = \Omega(t^6 / d) = o(1)$, a $1 - o(1)$ fraction of the sum comes from $C$.

Now we turn our attention to
\begin{equation*}
    C = \{ (S_1, S_2, \ldots, S_t) \in D : \bigcup\limits_{j=1}^t S_j \subseteq \{i : S_i = \emptyset\} \} .
\end{equation*}
Define $f$ on $D \setminus C$ by taking the first index $i$ such that $S_i \neq \emptyset$ and $i \in S_j$, and removing it from $S_j$. This will multiply the $j$th bound of $\Psi_{S_1, S_2, \ldots, S_t}$ by a factor of at least $M = \Omega(d / t^4)$ like before, while leaving the $i$th bound unchanged. Once again, this modification may send up to $N = t^2$ terms into a single one and $N / M = o(1)$, so the combined contributions of terms in $D \setminus C$ is an $o(1)$ fraction of the total sum.

Note that we already know the $S_i$ to be pairwise disjoint, so this implies that there are at least as many $i$ with $S_i = \emptyset$ as there are nonempty $S_j$. In other words, we are left only with terms with at least half of the $S_i$ empty. We now take
\begin{equation*}
    C = \{ (S_1, S_2, \ldots, S_t) \in D : |S_i| \leq 1 \} ,
\end{equation*}
and $f$ that takes the smallest $a < b \in S_i$, removes them both, and adds $a$ to $S_b$ instead. This will multiply the $i$th bound of $\Psi_{S_1, S_2, \ldots, S_t}$ by a factor of at least $\Omega(d^2 t^{-4})$, and the $b$th bound by at least $\Omega(t^{-4} d^{-1})$, all while leaving the $a$th bound unchanged. Each $S$ has at most $N = t^2$ preimages --- corresponding to a choice of $b$ and $i$, since $a$ will be determined by $S_b = \{ a \}$. Therefore $N / M = \mathcal{O}(t^{10} / d) = o(1)$.

Finally we consider
\begin{equation*}
    C = \{ (S_1, S_2, \ldots, S_t) \in D : \sum\limits_{i = 1}^t |S_i| = \lfloor{ t / 2 \rfloor} \} .
\end{equation*}
For $S \in D \setminus C$, we can find two empty $S_i, S_j$ such that $i, j$ are not in any other $S$. Let $f$ act by finding the first such $i < j$ and adding $i$ to $S_j$. This will multiply the $i$th bound of $\Psi_{S_1, S_2, \ldots, S_t}$ by a factor of at least $\Omega(d / t)$, and the $j$th bound by at least $\Omega(t^{-4} / t) = \Omega(t^{-5})$. Since there are at most $N = t$ options to reverse this process by choosing $S_j = \{i\}$, we have $N / M = \Omega(t^7 / d) = o(1)$.

To summarize, a $1 - o(1)$ fraction of the sum must come from terms $S$ such that $[t]$ can be paired into $(i, j)$ with $i < j, S_i = \emptyset, S_j = \{ i \}$, plus one possible unpaired index if $t$ is odd. Each such pair contributes
\begin{equation*}
    |\chi_{\emptyset} (p_i)| \cdot |\chi_{\{i\}} (p_j)| \leq (1-p) \cdot p \cdot \frac {1} {n} \left(1 + \mathcal{O} \left( \frac {t^2} {d} \right) \right)= \frac {p (1 - p)} {n} \left(1 + \mathcal{O} \left( \frac {t^2} {d} \right) \right) ,
\end{equation*}
times an additional $t^{-4}$ factor if $e_i, e_j$ are non-incident. Multiplying this over all such pairs, and noting that the product of all of the $(1 + \mathcal{O} (t^2 / d))$ terms is $1 + o(1)$, we get a bound $\Psi_{S_1, S_2, \ldots, S_t} \leq (1 + o(1)) t^{-4r} (p (1-p) / n)^{t / 2}$ with $r$ being the number of pairs $(i, j)$ with $e_i, e_j$ non-incident.

The question of bounding $|M_{\Gamma}|$ therefore boils down to another question, combinatorial in nature: Let $\Gamma$ be a sequence of $t$ edges $e_1, e_2, \ldots, e_t$, so that $\Gamma$ returns to a previously discovered vertex at most $m$ times. Then we would like to bound the number of ways to partition its edges to $t / 2$ pairs, so that exactly $r$ of those pairs have non-incident edges. We first bound the case of $r = 0$, in which we want all pairs to have incident edges.

If $m = 0$, then the only pairs of edges $e_i, e_j$ that could possibly be incident are the ones with $i + 1 \equiv j \pmod t$ or vice versa. Therefore, at most two pairings with the edges in each pair being incident are possible: $\{ (e_1, e_2), (e_3, e_4), \ldots, (e_{t - 1}, e_t) \}$ and $\{ (e_2, e_3), (e_4, e_5), \ldots, (e_{t - 2}, e_{t - 1}), (e_t, e_1) \}$.

If, $m > 0$, then each time $\Gamma$ returns to a previously visited vertex means that some edge $e_i$, and possibly also $e_{i+1}$, are incident to some of the edges $e_1, e_2, \ldots, e_{i - 2}$. We may first consider those edges $e_i, e_{i+1}$ for every edge $e_i$ which revisits a previously discovered vertex --- $2m$ edges in total, and pair each of them with some incident edges. There are at most $t^{2m}$ ways to do so, and then at most $2$ ways to partition the remaining edges into incident pairs by the same logic as before. Thus, the number of options to choose a pairing with $r = 0$ and arbitrary $m$ is at most $2 t^{2m}$.

Turning our attention to the case of arbitrary $r$, exactly $r$ pairs of edges may be non-incident. We may start by choosing those pairs, out of at most $t^{2r}$ options. But now, the remaining edges themselves form a sequence $\Gamma'$ of $t - 2r$ edges, and $\Gamma'$ also returns to a previously discovered vertex at most $m$ times. We must choose from it a pairing with $r = 0$ non-incident pairs, which we already know to be possible in at most $2 t^{2m}$ ways.

Therefore, for a walk $\Gamma$ that returns to a previously discovered vertex $m$ times, there are at most $2 t^{2m + 2r}$ pairings with $r$ non-incident pairs. This implies a bound
\begin{equation*}
    |M_\Gamma| \leq (2 + o(1)) [p (1-p) / n]^{t/2} \sum\limits_{r=0}^{\infty} t^{2m + 2r} t^{-4r} = (2 + o(1)) t^{2m} [p (1-p) / n]^{t/2}. \qedhere
\end{equation*}
\end{proof}

This implies a bound for arbitrary $\Gamma$ as well:

\begin{corollary}
\label{walk_contribution_bound}
Let $\log n \ll k \ll d^{1 / 10}$ and $d < c n$ and let $\Gamma$ be a closed walk of length $k$ with $m$ edges that return to previously discovered vertices, $t$ non-repeating edges, and $t_2$ distinct repeating edges. Then
\begin{equation*}
    |M_{\Gamma}| \leq (2 + o(1)) k^{2m} [p (1-p)]^{t_2 + t/2} n^{-t / 2} .
\end{equation*}
\end{corollary}

\begin{proof}
We begin by considering a sequence of edges $\Gamma'$, which is constructed by taking the $t$ edges which appear exactly once in $\Gamma$. Since $\Gamma'$ also returns to previously discovered vertices at most $m$ times, \Cref{non_repeating_bound} shows that $|M_{\Gamma'}| \leq (2 + o(1)) t^{2 m} [p (1 - p) / n]^t \leq (2 + o(1)) k^{2 m} [p (1 - p) / n]^t$.

But through the reasoning of \Cref{mixed_walks}, we may deduce from this a bound on $|M_{\Gamma}|$ of the entire walk:
\begin{equation*}
    |M_{\Gamma}| \leq (2 + o(1)) k^{2 m} [p (1-p) / n] ^ {t / 2} [p (1 - p)] ^ {t_2} = (2 + o(1)) k^{2 m} [p (1 - p)] ^ {t_2 + t / 2} n^{- t / 2} . \qedhere
\end{equation*}
\end{proof}

\section{Enumeration of Walks}
\label{section_enumeration}

In the previous section, a bound on the contribution $|M_{\Gamma}|$ of each closed walk $\Gamma$ has been given in terms of various parameters $k, m, t, t_2$ of the walk. Now, in order to obtain a bound on the total contribution of all walks, we need to determine how many closed walks exist given certain values of these parameters.

One can obtain such bounds using existing techniques. We will use similar methods to those of \cite{furedikomlos1981, vu2005}. To effectively apply them, we introduce new parameters: $b$ is the number of vertices $\Gamma$ visits, and $r$ is the number of revisits of a vertex coming from a non-repeating edge. Denote by $W_n(k, t, t_2, m, b, r)$ the number of non-lazy walks on $K_n$ with these parameters.

\begin{lemma}
\label{walk_enumeration}
For any valid $n, k, t, t_2, m, b, r$, we have
\begin{equation*}
    W_n(k, t, t_2, m, b, r) \leq n^b \binom {k} {2b - 2 - t + r} 2^{2b - 2 - t + r} b^{2k-4b+4+2t-2r}.
\end{equation*}
\end{lemma}

\begin{proof}
Consider walks of length $k$ on $b$ vertices, with $t$ non-repeating edges, $t_2$ distinct repeating edges, and $m$ edges that return to previously discovered vertices. Therefore, there are at most $t + t_2$ distinct edges, and $b \leq t + t_2 + 1 - m$.

The argument detailed in \cite{vu2005} then gives a bound on the number of such walks, and we also present it here briefly. Every new vertex is one of at most $n$ options, so there are at most $n^b$ choices of the list of vertices, in the order in which they are discovered.

Given this list, we classify the steps of the walk into:
\begin{itemize}
    \item \textit{Positive steps} reach a previously undiscovered vertex.
    \item \textit{Negative steps} traverse the edge of a previous positive step for the second time.
    \item \textit{Neutral steps} are neither positive nor negative --- for example, steps that traverse an edge for a third time, or steps on a non-repeating edge that reaches a previously visited vertex.
\end{itemize}

We now encode each walk using a codeword, that we want to be uniquely decodable given the parameters of the walk and the list of vertices. We begin by denoting positive steps by $``+"$, negative steps by $``-"$, and neutral steps by their second vertex. This is almost enough to fully decode the walk, but it is possible for a negative step to be incident to several positive edges that have only been traveled once so far, so that it can travel back on any of those edges.

In order to mend this, we introduce the notion of \textit{condensed codewords}. We can condense a codeword by repeatedly removing the substring $``+-"$ from the original encoding, as this substring can always be uniquely decoded. After that, it is shown in \cite{vu2005} that any remaining substring of consecutive $``-"$ signs can be uniquely decoded by writing down the vertex in which it ends. In other words, writing down those vertices will make our original codeword uniquely decodable. It is also shown in \cite{vu2005} that the number of extra vertices we will be writing down is at most the number of neutral edges.

Now, to calculate the number of such walks, observe that $r$ is precisely the number of neutral steps that are taken on non-repeating edges, and that the other $t - r$ non-repeating edges must all appear as positive steps in the codeword. Note that each such step must end in a previously discovered vertex, and therefore $r \leq m$. We thus need to consider the $k$ steps in our codeword, and choose which $b - 1$ of them will be positive, which $b - 1 - t + r$ will be negative, and which $k - 2b + 2 + t - r$ will be neutral.

This gives at most
\begin{equation*}
    \binom {k} {2b - 2 - t + r} 2^{2b - 2 - t + r} b^{k - 2b + 2 + t - r}
\end{equation*}
options to choose the $+, -$ and neutral symbols for each step. Also, at most $k - 2b + 2 + t - r$ negative steps require the extra end vertex notation, which costs $b$ for each such edge.

Thus,
\begin{equation*}
    W_n(k, t, t_2, m, b, r) \leq n^b \binom{k} {2b - 2 - t + r} 2^{2b - 2 - t + r} b^{2k-4b+4+2t-2r} . \qedhere
\end{equation*}
\end{proof}

Now, we have a bound on $W_n(k, t, t_2, m, b, r)$ and we know that each such walk contributes at most $M_{\Gamma} \leq (2 + o(1)) n^{-t / 2} k^{2 m} [p (1 - p)]^{t_2 + t / 2}$ to the expected trace. When combined, these two results should suffice to obtain a bound on the combined contributions of all possible walks $\Gamma$, thus proving \Cref{actual_main_result}.

\begin{proof}[Proof of \Cref{actual_main_result}]

Combining the number of walks from \Cref{walk_enumeration} with the bound on $M_{\Gamma}$ from \Cref{walk_contribution_bound}, the contribution of all closed walks with given $k, t, t_2, r, b, m$ is at most
\begin{equation*}
(2 + o(1)) n^{b - t / 2} \binom {k} {2b - 2 - t + r} 2^{2b - 2 - t + r} k^{2m} [p (1 - p)]^{t_2 + t / 2} b^{2k - 4b + 4 + 2t - 2r}.
\end{equation*}

Fix some values for $k, t, t_2$ and let $m, b, r$ vary. We want to find the $m, b, r$ that maximize the term above. Since they only effect the
\begin{equation*}
    n^b \binom {k} {2b - 2 - t + r} 2^{2b} k^{2m} b^{2k - 4b - 4 + 2t - 2r}
\end{equation*}
part, and since $n^b$ grows fast with $b$, we may hope that maximizing $b$ is the most effective strategy.

Indeed, increasing $b$ by $1$ may decrease the other terms by a $\mathcal{O}(k^6)$ factor in the worst case, so the sum of terms with some given $k, t, t_2, m, r$ is dominated by the term with $b = t_2 + t - m + 1$.

Now, we may further increase $b$ by decreasing $m$ (and possibly $r$ if necessary, since $r \leq m$) by $1$. This will still increase $n^b$ by a factor of $n$, and may still only decrease other factors by $\mathcal{O}(k^6)$. Thus, the sum of terms with given $k, t, t_2$ is dominated by the term with $m = r = 0$ and $b$ maximal, which is $(2 + o(1)) n^{t_2 + t / 2 + 1} 2^{2 t_2 + t} [p (1-p)]^{t_2 + t/2} = (2 + o(1)) n [4 n p (1 - p)]^{t_2 + t/2}$.

Since $4 n p (1 - p) > 1$, this is maximized when $t_2 + t / 2$ is maximal. Since there are $k - t$ edges in the repeating part, and each of them appears at least twice, then $t_2 \leq (k - t) / 2$ --- which gives $t_2 + t / 2 \leq k / 2$. Therefore, the total contribution of terms with given $k, t, t_2$ is at most $(2 + o(1)) n [4 n p (1-p)]^{k / 2}$.

Summing this over the possible values of $t, t_2$, which are between $0$ and $k$, we get a bound of
\begin{equation*}
    (2 + o(1)) (k + 1) ^ 2 n [4 n p (1-p)]^{k/2}
\end{equation*}
on the contribution of all closed walks $\Gamma$ of length $k$, which is just the expected trace. As claimed in the statement of \Cref{actual_main_result}, this is no more than $P(n) [4 d (n - d) / n]^{k/2}$ for some polynomial $P(n)$ --- for example, $P(n) = n^2$ would work, since $(k + 1)^2 \ll n$.
\end{proof}

Thus, \Cref{actual_main_result} is proven, thereby proving \Cref{main_theorem} as well.

\section{Acknowledgements}
The author wishes to thank Michael Krivelevich for his guidance and support. He is also very grateful to Eden Kuperwasser and Yotam Shomroni for their helpful comments and suggestions.

\setcounter{biburllcpenalty}{7000}
\setcounter{biburlucpenalty}{8000}

\printbibliography

\appendix
\section{Estimating the Fixed Points of the Operators}
\label{first_order_estimate}

We present here the full proof of \Cref{p_estimate}. For brevity, given $\mathbf{d}$ and $\mathcal{A}$, we write $d_a$ for $\mathbf{d}(a)$ and $n_a$ for $|\mathcal{A}(a)|$. We need to prove the following:

\begin{proposition}
\begin{equation}
\label{p_estimate_plugin}
    \tilde{P}_{\mathbf{d}, \mathcal{A}} (ab) = \frac {(n - 1) d_a d_b} {d n_a n_b} \left(1 + \mathcal{O}\left( \frac {t^2} {n d} \right) \right) ,
\end{equation}
\begin{equation}
\label{y_estimate_plugin}
    \tilde{Y}_{\mathbf{d}, \mathcal{A}} (abc) = \frac {(n - 1)^2 d_a d_b (d_b - 1) d_c} {d^2 n_a n_b (n_b - 1) n_c} \left(1 + \mathcal{O}\left( \frac {t^2} {n d} \right) \right) ,
\end{equation}
constitute a fixed point of the operators $\mathcal{P}, \mathcal{Y}$.
\end{proposition}

\begin{proof}
We begin by calculating $\mathcal{B}(\tilde{P}, \tilde{Y})$, as defined in \Cref{b_operator}. It is composed of two sums, and we begin by considering
\begin{equation*}
    \sum\limits_{c \in \mathcal{A}(a) \cap \mathcal{A}(b)} Y_{\mathbf{d}, \mathcal{A}} (acb) .
\end{equation*}
Note that each summand is $p^2 (1 + \mathcal{O}(t / d))$, and at most $2t$ vertices will have either $d_c < d$ or $n_c < n-1$. Therefore, we may freely assume that the sum does not involve any such $c$, as they will only add a negligible error.

Plugging in \Cref{y_estimate_plugin}, we therefore get that the sum is equal to the number of terms $|\mathcal{A}(a) \cap \mathcal{A}(b)|$, times $Y_{\mathbf{d}, \mathcal{A}} (acb)$ for any vertex $c$ with $d_c = d, n_c = n - 1$. This yields an estimate of 
\begin{equation*}
    |\mathcal{A}(a) \cap \mathcal{A}(b)| \cdot \frac{(n-1)(d-1) d_a d_b} {(n-2)d n_a n_b} + \mathcal{O} \left( \frac {d t^2} {n^2} \right) .
\end{equation*}
A similar argument shows that the other sum is
\begin{equation*}
    |\mathcal{A}(a) \setminus \mathcal{A}(b)| \cdot \frac {d_a} {n_a} + \mathcal{O} \left( \frac {t^3} {n^2} \right) .
\end{equation*}

This gives us an estimate on $\mathcal{B}(\tilde{P}, \tilde{Y})$, which we are now ready to plug into the various summands of \Cref{p_operator},
\begin{equation*}
    \frac {\mathbf{d}(c) - B_{\mathbf{d} - \mathbf{e}_a - \mathbf{e}_b, \mathcal{A}} (ca)} {\mathbf{d}(a) - B_{\mathbf{d} - \mathbf{e}_b - \mathbf{e}_c, \mathcal{A}} (ac)} \frac {1 - P_{\mathbf{d} - \mathbf{e}_b - \mathbf{e}_c, \mathcal{A}}(bc)} {1 - P_{\mathbf{d} - \mathbf{e}_a - \mathbf{e}_b, \mathcal{A}} (ab)} .
\end{equation*}

For start, they are always $1 + \mathcal{O}(t / d)$, so we can similarly assume that $d_c = d, n_c = n - 1$, and $c \in \mathcal{A}(a)$. In that case, the numerator has
\begin{equation*}
\begin{split}
    \mathbf{d}(c) - B_{\mathbf{d} - \mathbf{e}_a - \mathbf{e}_b, \mathcal{A}} (ca) & = d - (n_a - 1) \frac {(d - 1) (d_a - 1)} {(n - 2) n_a} - (n - n_a) \frac {d} {n - 1} + \mathcal{O} \left( \frac {d t^2} {n^2} \right) \\
    & = d - \frac {(d_a - 2) d} {n - 1} - \frac {(n - n_a) d} {n - 1} + \mathcal{O} \left( \frac {t^2} {n} \right) \\
    & = \frac {d} {n-1} (n_a + 1 - d_a) + \mathcal{O} \left( \frac {t^2} {n} \right) .
\end{split}
\end{equation*}
For the denominator, we similarly get
\begin{equation*}
    \mathbf{d}(a) - B_{\mathbf{d} - \mathbf{e}_b - \mathbf{e}_c, \mathcal{A}} (ac) = \frac {d_a} {n - 1} (n - d) + \mathcal{O} \left( \frac {t^2} {n} \right) .
\end{equation*}
Putting together these two equations, we get
\begin{equation*}
    \frac {\mathbf{d}(c) - B_{\mathbf{d} - \mathbf{e}_a - \mathbf{e}_b, \mathcal{A}} (ca)} {\mathbf{d}(a) - B_{\mathbf{d} - \mathbf{e}_b - \mathbf{e}_c, \mathcal{A}} (ac)} = \frac {d (n_a + 1 - d_a)} {d_a (n - d)} + \mathcal{O} \left( \frac {t^2} {n d} \right) = \frac {d (n_a - d_a)} {d_a (n - 1 - d)} + \mathcal{O} \left( \frac {t^2} {n d} \right) .
\end{equation*}

We now need to multiply this by $(1 - \tilde {P}_{\mathbf{d} - \mathbf{e}_b - \mathbf{e}_c, \mathcal{A}}(bc)) / (1 - \tilde  {P}_{\mathbf{d} - \mathbf{e}_a - \mathbf{e}_b, \mathcal{A}} (ab))$. To simplify, we will first consider the terms not involving $a$, and then those involving them.

The part without $a$ is
\begin{equation*}
\begin{split}
    \frac {d} {n - 1 - d} (1 - \tilde {P}_{\mathbf{d} - \mathbf{e}_b - \mathbf{e}_c} (bc)) & = \frac {d} {n - 1 - d} \left( 1 - \frac {(n-1) (d_b - 1) (d - 1)} {d n_b (n - 1)} \right) + \mathcal{O} \left( \frac {t^2} {n^2} \right) \\
    & = \frac {d} {n - 1} \left(1 + \frac {d^2 n_b - d (n - 1) (d_b - 2)} {d n_b (n - 1 - d)} \right) + \mathcal{O} \left( \frac {t^2} {n^2} \right) ,
\end{split}
\end{equation*}
as can be shown by a straightforward calculation. The remaining terms are analogous,
\begin{equation*}
    \left( \frac {d_a} {n_a - 1 - d_a} (1 - \tilde {P}_{\mathbf{d} - \mathbf{e}_a - \mathbf{e}_b} (ab)) \right) ^ {-1} = \left( \frac {d_a} {n_a} \left(1 + \frac {d_a d n_b - d_a (n - 1) (d_b - 2)} {d n_b (n_a - d_a)} \right) + \mathcal{O} \left( \frac {t^2} {n^2} \right) \right) ^ {-1} .
\end{equation*}
The ratio between the inner terms is
\begin{equation*}
    \left(1 + \frac {d_a d n_b - d_a (n - 1) (d_b - 2)} {d n_b (n_a - d_a)} \right) / \left(1 + \frac {d^2 n_b - d (n - 1) (d_b - 2)} {d n_b (n - 1 - d)} \right) = 1 + \mathcal{O} \left( \frac {t^2} {n d} \right),
\end{equation*}
so the entire summand turns out to be 
\begin{equation*}
    \frac {d n_a} {d_a (n - 1)} \left(1 + \mathcal{O} \left( \frac {t^2} {n d} \right) \right) .
\end{equation*}

This is now summed over the $n_b$ vertices in $A^{*}(b)$. Finally, taking the inverse and multiplying by $d_b$, we get the required estimate
\begin{equation*}
    \frac {(n-1) d_a d_b} {d n_a n_b} \left(1 + \mathcal{O} \left( \frac {t^2} {n d} \right) \right) ,
\end{equation*}
which is precisely what we had in \Cref{p_estimate_plugin} to begin with.

We still need to plug $\tilde{P}, \tilde{Y}$ into $\mathcal{Y}$ to show that this estimate is a fixed point for this operator as well. We begin by observing that $\tilde{Y}_{\mathbf{d} - \mathbf{e}_a - \mathbf{e}_b, \mathcal{A}} (abc)$ can be written as a product of $P$ terms as $\tilde{P}_{\mathbf{d} - \mathbf{e}_a - \mathbf{e}_b, \mathcal{A}} (bc) \tilde{P}_{\mathbf{d} - 2\mathbf{e}_a - 2\mathbf{e}_b, \mathcal{A}} (ab) \left(1 + 1 / n + \mathcal{O}(t / n^2) \right)$.

Plugging this into the definition of $\mathcal{Y}$ from \Cref{y_operator}, we get
\begin{equation*}
    \mathcal{Y}(\tilde{P}, \tilde{Y})_{\mathbf{d}, \mathcal{A}} (abc) = \tilde{P}_{\mathbf{d}, \mathcal{A}} (ab) \tilde{P}_{\mathbf{d} - \mathbf{e}_a - \mathbf{e}_b, \mathcal{A}} (bc) \frac {1 - \tilde{P}_{\mathbf{d} - \mathbf{e}_a - 2\mathbf{e}_b, \mathcal{A}} (ab) (1 + 1 / n + \mathcal{O}(t^2 / n))} {1 - \tilde{P}_{\mathbf{d} - \mathbf{e}_a - \mathbf{e}_b, \mathcal{A}} (ab)} ,
\end{equation*}
so we just need to prove that
\begin{equation*}
    \frac {1 - \tilde{P}_{\mathbf{d} - \mathbf{e}_a - 2\mathbf{e}_b, \mathcal{A}} (ab) (1 + 1 / n)} {1 - \tilde{P}_{\mathbf{d} - \mathbf{e}_a - \mathbf{e}_b, \mathcal{A}} (ab)} = 1 + \frac {1} {n} + \mathcal{O} \left( \frac {t^2} {n d} \right).
\end{equation*}
This can be verified by a straightforward calculation. Indeed, plugging in our estimates $\tilde{P}$ we get
\begin{align*}
    \frac {d n_a n_b - n (d_a - 2) (d_b - 1)} {d n_a n_b - (n - 1) (d_a - 1) (d_b - 1)} + \mathcal{O} \left( \frac {t^2} {n d} \right) & = 1 + \frac {n d_b - d_a d_b} {d n_a n_b - (n - 1) (d_a - 1) (d_b - 1)} + \mathcal{O} \left( \frac {t^2} {n d} \right) \\
    & = 1 + \frac {d (n - d)} {d (n - d) n} + \mathcal{O} \left( \frac {t^2} {n d} \right) = 1 + \frac {1} {n} + \mathcal{O} \left( \frac {t^2} {n d} \right) . \qedhere
\end{align*}
\end{proof}

\end{document}